\documentclass[12pt]{elsarticle}
\usepackage{latexsym,amssymb,amsfonts,amsmath,amsthm,graphicx}
\def\beq{ \begin{equation} }
\def\eeq{ \end{equation} }
\def\beqx{ \begin{equation*} }
\def\eeqx{ \end{equation*} }
\def\beqa{\begin{eqnarray}}
\def\eeqa{\end{eqnarray}}
\def\beqax{\begin{eqnarray*}}
\def\eeqax{\end{eqnarray*}}
\def\mn{\medskip\noindent}
\def\eopt{\hfill$\square$}

\numberwithin{equation}{section}

\newtheorem{theorem}{Theorem}

\newtheorem{lemma}{Lemma}[section]

\newtheorem{proposition}{Proposition}[section]
\newtheorem{corollary}{Corollary}

\newcommand{\ep}{\varepsilon}
\newcommand{\n}{\noindent}

\newcommand{\R}{\mathbb{R}}

\begin{document}

\title{Evolution in predator-prey systems}
\author{Rick Durrett\fnref{fn1}}
\ead{rtd1@cornell.edu}
\author{John Mayberry\corref{cor1}\fnref{fn2}}
\ead{jm858@cornell.edu}
\address{Department of Mathematics, Cornell University, Ithaca, New York 14853, USA}

%\date{\today}
\cortext[cor1]{Corresponding author: Tel.: +1 607 255 8262;  fax: +1 607 255 7149}
\fntext[fn1]{Partially supported by NSF grant DMS 0704996 from the probability program.}
\fntext[fn2]{Partially supported by NSF RTG grant DMS 0739164.}

\begin{abstract}
We study the adaptive dynamics of predator-prey systems modeled by a dynamical system in which the traits of predators and prey are allowed to evolve by small mutations. When only the prey are allowed to evolve, and the size of the mutational change tends to 0, the system does not exhibit long term prey coexistence and the trait of the resident prey type converges to the solution of an ODE. When only the predators are allowed to evolve, coexistence of predators occurs. In this case, depending on the parameters being varied, we see (i) the number of coexisting predators remains tight and the differences in traits from a reference species converge in distribution to a limit, or (ii) the number of coexisting predators tends to infinity, and we calculate the asymptotic rate at which the traits of the least and most ``fit'' predators in the population increase. This last result is obtained by comparison with a branching random walk killed to the left of a linear boundary and a finite branching-selection particle system.
\end{abstract}
\begin{keyword}
predator-prey \sep adaptive dynamics \sep coexistence \sep Lotka-Volterra equations  \sep branching random walk \sep branching-selection particle system

\MSC \, primary 92D15 \sep 92D25; \,  secondary 60J60 \sep 60K35
\end{keyword}
\maketitle
\section{Introduction} \label{intro}
The rapidly developing field of adaptive dynamics emphasizes the combined effects of evolution and ecological interactions on population dynamics. To describe the general framework of this theory, consider a population of individuals, each attached with a trait or strategy $x$ that characterizes its ability to survive and propagate. The current distribution of traits governs the population dynamics by describing interactions between different individuals and their environment. Underlying this ecological process, is a slower mutational process that occasionally introduces a new trait into the population. Understanding the interplay of these fast ecological and slow mutational processes is the primary objective of adaptive dynamics. Foundations of the theory were laid in early 1990's by Hofbauer and Sigmund \cite{HSA}, Metz, Nisbet, and Geritz \cite{MNG}, and Dieckmann, Marrow, and Law \cite{DML} and focused on macroscopic models, i.e., ODE models describing large population limits. In their 1996 paper \cite{DL}, Dieckmann and Law suggest that: ``A proper mathematical theory of evolution should be dynamical...The dynamics ought to be underpinned by a microscopic theory''. A rigorous foundation for microscopic models via multi-type branching processes has now been developed (see for example Champagnat and Lambert \cite{CL} and Champagnat, Ferri\`ere, and M\'el\'eard \cite{CFM}).

In this paper, we study the dynamics of coexistence that arise as a consequence of introducing rare, small mutations into a model for predator-prey interactions. The novelty of our work lies in the establishment of coexistence of a large number of types. This phenomenon is known as polymorphic evolution. Two other notable examples of polymorphic evolution in the adaptive dynamics literature are: (i) evolutionary branching (first described in Geritz et al.~\cite{GM} and more recently studied in Champagnat and M\'el\'eard \cite{CM} from a microscopic perspective) which describes coexistence of types with diverging traits and (ii) the Tube Theorem of Geritz et al.~\cite{GK}, \cite{Ger} where coexistence of a resident and invading type with similar survival strategies occurs inside of a ``tube'' in which the sum of the invader and resident population sizes stays close to the former resident attractor. The situation we encounter more closely resembles the second scenario: types with very similar traits can coexist. In our model, this is due to the fact that interspecific competition (competition with individuals of different types) has less of an effect than intra-specific competition (competition with individuals of the same type).

%Second, evolutionary branching has only been studied in models which incorporate competitive interactions (species competing for resources), but not predator-prey type interactions. The difference, is of course, that in competitive models, each species is hurt by the presence of other types whereas in predator-prey models, predators are helped by the presence of more prey. Results on dimorphism are similar to our results in the first respect (coexistence of similar types), but still differ in respect to the second (including predator-prey interactions). Our model provides the first example of coexistence in an adaptive dynamics type model with predator-prey interactions.

%the context of a population of types competing for the resources of an environment without

Since our focus will be on the dynamics of the random process of types that emerges from our underlying mutational process, we shall take a macroscopic perspective of population dynamics, using a Lotka-Volterra system of ODE's to describe predator-prey interactions. In particular, we will suppose that if we have $M$ prey types and $N$ predator types, then the densities are governed by the ODE's
\begin{eqnarray}
\frac{du_i}{dt} & =& u_i\left( \beta_i (1-\sum_k u_k) - 1 - \sum_j \alpha_{i,j} v_j \right) \nonumber \\
\frac{dv_j}{dt} & =& v_j\left( \sum_i \alpha_{i,j} u_i - \delta_j -  v_j \right) \label{LV}
\end{eqnarray}
where the $u_i$, $1\le i \le M$ are the densities of the prey, and the $v_j$, $1 \le j \le N$ are the densities of the predators. Our main interest is the effect of small mutations in the resident types on the equilibrium behavior of (\ref{LV}). While the co-evolutionary case in which both predator and prey are allowed to vary is certainly of interest and can lead to exotic behavior (see, for example, Dieckmann et al.~\cite{DML} and Dercole et al.~\cite{DIR}), we will here only consider the two cases of fixed predator/evolving prey and evolving predator/ fixed prey. Such examples are also of interest and have been studied in laboratory experiments (see, for example, Jones and Ellner \cite{JE}).

Following the usual approach in adaptive dynamics, we shall assume that mutations take place on a much slower time scale than the population dynamics reach equilibrium. In particular, suppose that we are considering predator evolution and we currently have $k$ predator types (which we call the resident types) and one prey type coexisting in equilibrium. We introduce a small density of a new type of predator, called the mutant type, with traits chosen according to a specified mutation distribution, and let the densities evolve according to (\ref{LV}) until a possibly new equilibrium is reached before introducing the next mutant type into the population. By traits, we mean the parameters in (\ref{LV}) (birth, death, and consumption rates) that characterize each individuals ability to survive and propagate. We then repeat this process, using only those predators which could coexist at the previous step. In this way, we obtain a Markov chain of resident types with transitions determined by the outcomes of the ecological interactions between the previous residents and the mutant. Once we have introduced some preliminary results for the ODE, we will formulate this process more precisely.

%This ``toy model'' can be thought of as an embedding (at the times of mutation) of a more general model in which mutations occur at some small rate $\gamma$ and we look at the limiting behavior of \eqref{LV} on the timescale $t/\gamma$ as $\gamma \to 0$.

Our evolutionary algorithm leads to a variety of different scenario depending on the underlying parameters being varied.
\begin{description}
\item{(a)} If we have a single, fixed predator and allow mutations in the prey's traits $(\alpha,\beta)$, then prolonged coexistence of prey does not occur and the traits of the resident prey converge to the solution of an ODE (see Theorem \ref{preymain}).
\item{(b)} If we have a single, fixed prey and allow the consumption rate $\alpha$ of predators to evolve, then coexistence of predators occurs, but the number of coexisting predators remains tight and the differences of the parameters from a reference type converge in distribution to a limit (see Theorem \ref{deltafixedmain}).
\item{(c)} If we have a single, fixed prey and allow the death rate $\delta$ of predators to evolve, the number of coexisting predators tends to infinity and we can calculate the speed at which the traits of the least and most fit predators in the population increase (see Theorem \ref{alphafixedmain}).
\end{description}
In all three cases, our results are more mathematically interesting than biologically relevant since in (b), for example, the consumption rate of all predators currently present in the population increases without bound.

The remainder of this section is dedicated to statements of these results and some conjectures for future research. Proofs of the three results in (a) - (c) above will be contained in Sections \ref{preysec}-\ref{alphafixedsec}.

\subsection{Prey Evolution} \label{preysecintro}

We begin with the case in which we have a single predator with fixed death rate $\delta > 0$ and we allow prey types to evolve. Throughout the remainder of this subsection when we refer to \eqref{LV}, we shall always assume that $N=1$ and let $v$ denote the density of our single predator. Prey types are characterized by their two dimensional trait vectors $y = (\alpha,\beta) \in \R^2$ and we say that prey types $y_1,...,y_M$ can \emph{coexist} with the predator if whenever we run \eqref{LV} started from positive initial densities, $v(0), u_i(0) > 0$, the densities $v(t)$, $u_i(t)$ remain bounded away from zero for all time. Our first step is to discuss criteria for coexistence. We use the notation $u = (u_1,u_2,...,u_M)$ and
\begin{align*}
\Gamma_{M,1} &= \{(u,v) \in \R^{M+1}: v, u_i \geq 0 , \, \quad \sum u_i \leq 1\} \\
\Gamma_{M,1}^{J,+}&= \{(u,v) \in \Gamma_{M,1}: v, u_i > 0 , \, \forall i \in J \}
\end{align*}
for any $J \subset\{1,2,...M\}$. If $J =\{1,...,M\}$, we simply write $\Gamma_{M,1}^+ = \Gamma_{M,1}^{J,+}$. Note that $\Gamma_{M,1}$ is invariant under \eqref{LV}. Here and elsewhere, we shall use $|\cdot|$ to denote the cardinality of a finite set, the absolute value of a real number, and the Lebesgue measure of a set of real numbers depending on the context.

\begin{proposition} \label{preyODEmain}
For all prey $y_i \in \R^2$, $i \leq M$, with different birth rates, \eqref{LV} has an explicitly calculable equilibrium $\sigma = (\sigma_1,...,\sigma_M,\sigma_{M+1}) \in \Gamma_{M,1}$ which is globally attracting on $\Gamma_{M,1}^+$. Furthermore, if $J_\sigma = \{i \leq M: \sigma_i >0\}$, then $|J_\sigma| \leq 2$ and $\sigma$ is globally attracting on $\Gamma_{M,1}^{J_\sigma,+}$ as well.
\end{proposition}

Proposition \ref{preyODEmain} follows from Lemmas \ref{2prey1pred} and \ref{3prey1pred} in Section \ref{preysec}. Since we need explicit formulas for the equilibria of \eqref{LV}, we will need to redo some standard results on Lotka-Volterra equations (see, for example, Chapters 13 and 15 of Hofbauer and Sigmund \cite{HS} and Chapter 3 in Takeuchi \cite{Tak}) to prove our results.

We are now ready to describe the Prey Evolutionary Process (Prey EP). This process is a continuous time Markov jump process which keeps track of the current resident prey types in the population. Proposition \ref{preyODEmain} tells us that we will never have more than two coexisting prey types so at time $t$, the state of the Prey EP is $\mathbf{Y}(t) = (Y_1(t),Y_2(t)) \in \R^2 \times \R^2$. For initial conditions, we set $Y_2(0) = (0,0)$, i.e., the second prey species is initially absent, and choose any $Y_1(0) = (\alpha(0),\beta(0))$ satisfying
$$
Y_1(0) \in \mathcal{V} \equiv \left\{(\alpha,\beta) \in R^2_+: \beta > \frac{\alpha}{\alpha-\delta} > 1\right\}.
$$
The reason for this choice of $Y_1(0)$ is that if $M=1$, the globally attracting equilibrium described in Proposition \ref{preyODEmain} satisfies $\sigma_1, \sigma_2 >0$ if and only if the prey type has traits in $\mathcal{V}$ (see \eqref{viable} in Section \ref{preysec}). As long as $Y_1(t) \neq (0,0)$, mutational events occur at rate 1 and after a mutation, the transitions for $\mathbf{Y}(t)$ are determined by the following procedure. We pick one of the non-zero $Y_i(t-)$, $i=1,2$ at random and choose $Y_{new}= (\alpha_{new},\beta_{new})$ uniformly from $B_\ep(Y_i(t-))$, the ball of radius $\ep$ around $Y_i(t-)$. If $Y_2(t-) \neq (0,0)$, let $\sigma$ be the equilibrium obtained in Proposition \ref{preyODEmain} when $M=3$ and the prey have traits $y_1,y_2,y_3 = Y_1(t-), Y_2(t-), Y_{new}$. If $Y_2(t-) = (0,0)$, then let $\sigma$ be the equilibrium obtained in Proposition \ref{preyODEmain} when $M=2$ and the prey have traits $y_1, y_2 = Y_1(t-), Y_{new}$. Note that since the probability of inserting a mutant with the same birth rate as one of the residents is 0, we do not have to worry about the exceptional case where Proposition \ref{preyODEmain} does not apply. If $|J_\sigma|=2$, then we set $Y_1(t)$ and $Y_2(t)$ equal to the parameter values of the two prey with positive equilibrium densities. If $|J_\sigma|=1$, then we set $Y_1(t)$ equal to the parameter values of the single prey with positive equilibrium density and take $Y_2(t)=(0,0)$. If $|J_\sigma|=0$, we set $Y_1(t), Y_2(t)=(0,0)$ and the process enters an absorbing state. We say that the population is monomorphic when $Y_2(t) = (0,0)$ and refer to the events where $Y_2(t)$ jumps from $(0,0)$ as coexistence events.

%
%Letting $(1) = \min\{i: \sigma_i >0\}$ and $(2)=\min\{ i > (1): \sigma_i > 0\}$, the new state of the Prey EP is then
%\beqx
%(Y_1(t),Y_2(t)) = \begin{cases}
% (y_{(1)}, y_{(2)}) & \text{ if } |J_\sigma| = 2\\
%(y_{(1)}, 0) & \text { if } |J_\sigma| = 1 \\
%(0,0) & \text{ if } |J_\sigma|=0.
%\end{cases}
%\eeqx
%Note that by definition of the Prey EP, $(0,0)$ is an absorbing state.

Our first main result says that in the small mutation limit, the population is essentially monomorphic.

\begin{theorem} \label{preymain}
Let $T >0$. As $\ep \to 0$, $Y_1^\ep(t) \equiv Y_1(t/\ep) \to y_1(t)$ in probability uniformly on $[0,T]$. $y_1(t)$ is the unique solution to the ODE
\beq \label{FAD}
\frac{dy_1}{dt} = \frac{2}{3\pi} {\cal N}(y_1(t))
\eeq
with initial conditions $y_1(0) = Y_1(0)$ where $\mathcal{N}(\cdot)$ is explicitly calculable, see (\ref{vectorN}). Furthermore, if we let $Y_2^\ep(t)= Y_2(t/\ep )$ and
$$N_t^\ep \equiv |\{s \leq t: Y_2^\ep(s-) =0, \, Y_2^\ep(s)\neq 0\}|$$
$t \leq T$, be the number of times $Y_2$ jumps from 0 before time $t/\ep$, then as $\ep \to 0$,
$$N^\ep \Rightarrow N$$
where $N$ is a nonhomogeneous Poisson Process on $[0,T]$ and $``\Rightarrow''$ denotes convergence in distribution.
\end{theorem}

We will prove Theorem \ref{preymain} in Section \ref{preysec}. The proof reveals that $|\{t\leq T: Y_2^\ep(t\wedge \tau) > 0\}| \to 0$ (see \eqref{y2neg}) which justifies our earlier claim that prolonged coexistence of prey does not occur. The constant on the right hand side of \eqref{FAD} is $EY^+$ when $(X,Y)$ is chosen at random from the ball of radius 1 and appears due to our choice of mutation distribution. (\ref{FAD}) is essentially a special case of the ``Canonical Equation of Adaptive Dynamics," see (6.2) in Dieckmann and Law \cite{DL} or (1) in Champagnat and Lambert \cite{CL}. In words, evolution proceeds in the direction of fastest increase in fitness. We do not have an explicitly defined fitness, but the infinitesimal drift in the traits is perpendicular to the region of values that cannot invade the resident.

The limiting ODE in Theorem \ref{preymain} is not biologically sensible because the prey birth rate increases without bound. This could be remedied by restricting the permitted values of $(\alpha,\beta)$ to a curve, or making them functions of other underlying parameters (see for example Dieckmann et al.~\cite{DML}, where traits depend on the ``body size'' of the predator and prey). However, since our main interest in including Theorem \ref{preymain} is as a contrast to the results on predator evolution below, we do not here endeavor to carry out the details of this more realistic scenario.

\subsection{Predator Evolution} \label{predatorsecintro}

We now consider the case where we have a single prey with fixed birth rate $\beta > 1$ and density $u$ ($M=1$ in \eqref{LV}), but we allow our predators to evolve. Predators are characterized by their two dimensional trait vector $x = (\alpha,\delta) \in \R_+^2$. As in the previous section, the first step is to develop a criteria for determining coexistence of multiple predators.  The next result, which is proved in Section \ref{predatorsec}, tells us that this can be done by checking a simple algebraic condition. In order to state the result, we define the \emph{characteristic ratio} of predator $x= (\alpha,\delta)$ as $\ell = \alpha/\delta$ and use the notation
\begin{align*}
\Gamma_{1,N} &= \{(u,v_1,...,v_N) \in \R^{N+1}: 0 \leq u \leq 1, v_i \geq 0\} \\
\Gamma_{1,N}^{+} &= \{(u,v_1,...,v_N) \in \R^{N+1}: 0 < u \leq 1, v_i > 0\}.
\end{align*}
Note that $\Gamma_{1,N}$ is invariant for \eqref{LV}.

\begin{proposition} \label{odemain}
For any $N \geq 1$, \eqref{LV} has a unique, globally attracting equilibrium $\sigma$ for initial densities in $\Gamma_{1,N}^+$. This equilibrium has the following characterization. Suppose that predators $x_1,...,x_N$ are ordered by increasing characteristic ratios. Then the globally attracting equilibrium has a positive $i^{th}$ component if and only if $i=1,...,m$ where $m \leq N$ is the largest value of $k$ satisfying the condition:
\begin{equation} \label{comain}
 \sum_{j=1}^k \alpha_j^2 (\ell_k - \ell_j) < r - \beta \ell_k
\end{equation}
and $r=\beta-1 > 0$ is the intrinsic growth rate of the prey.
\end{proposition}

We prove Proposition \ref{odemain} in Section \ref{predatorsec}. The definition for the Predator Evolution Process (Predator EP) is similar to the definition of the Prey EP except the state space is now $(\R^2)^\mathbb{N}$ as there is no limit on the number of predators that can coexist. The state of the process at time $t$ is
$$\mathbf{X}(t) = ((\alpha_1(t),\delta_1(t)),(\alpha_2(t),\delta_2(t)),\dotso).$$
For initial conditions, we choose $(\alpha_1(0), \delta_1(0))$ so that $ \alpha_1(0)r /\beta > \delta_1(0)$ and set $\alpha_k(0), \beta_k(0) = 0$ for all $k \geq 2$. Our choice of $\alpha_1(0),\delta_1(0)$ guarantees that the initial predator can coexist with the prey (see Section \ref{predatorsec}). As long as $(a_1(t), \beta_1(t)) \neq (0,0)$, mutations occur at rate 1 and if a mutation occurs at time $t$, transitions are described by the following rules. Define $N_t = \max\{i: (\alpha_i(t-),\beta_i(t-)) \neq (0,0)\}$, choose one of $(\alpha_i(t-),\delta_i(t-))$, $i \leq N_t$ at random, and introduce a new mutant predator with traits $\alpha_{new} = \alpha_i + \ep  U_1$ and $\delta_{new} = \delta_i e^{\ep U_2}$ where $U_1, U_2 \sim$ Uniform[-1,1]. We then order the predators $(\alpha_1(t-),\beta_2(t-)),\dotso (\alpha_{N_t}(t-),\beta_{N_t}(t-)),$ $(\alpha_{new},\beta_{new})$ by increasing characteristic ratios and call the ordered traits $x_1,$ $\dotso,$ $x_{N_t+1}$. Applying Proposition \ref{odemain} with predators $x_1,...,x_{N_t+1}$, we obtain a globally attracting equilibrium $\sigma$ with $\sigma_i > 0$ if and only if $i \leq m$ for some $m \leq N_t+1$. We then set $(\alpha_i(t),\beta_i(t)) = x_i$ for all $i \leq m$ and set $(\alpha_i(t),\beta_i(t)) = (0,0)$ for all $i > m$.

Simulations suggest that we see a growing cloud of coexisting predators with some limiting shape and all predators have consumption rates $\alpha$ going off to infinity and $\log(\ell)$ going to $-\infty$ (see Figures \ref{2Dscatter} and \ref{2Dnumber}). We were not able to analyze the two dimensional system so we will specialize to the two cases where only $\alpha$ or $\delta$ varies and the other remains fixed.

\subsection{The Alpha Predator Evolution Process} \label{deltafixedintro}

In this section, we assume that $\delta=1$ remains fixed and allow for mutations in $\alpha$. Note that the characteristic ratio of a predator is now $1/\alpha$ so that ordering predators by increasing characteristic ratios is equivalent to ordering predators by decreasing consumption rates. We shall prove our results for a discrete time version of the Predator EP in which the $n^{th}$ mutation occurs at time $n$. Once this is done, it is straightforward to generalize the result to continuous time.

To more precisely describe the version we study, suppose that at time 0, we have a single predator $\alpha(0)$ which can coexist with the prey. If at time $n$, we have $N_n$ predators $\alpha_1(n),...,\alpha_{N_n}(n)$ in decreasing order that satisfy \eqref{comain} with $k = N_n$, then at time $n+1$ we choose one of the $\alpha_j(n)$, $j \leq N_n$ at random, introduce a mutant with trait $\alpha_{new} = \alpha_j(n) + \ep U$, $U \sim$ Uniform[-1,1], and then use \eqref{comain} to decide on the state of the process at time $n+1$ in the same manner as we did for the Predator EP. We shall refer to this process as the Alpha Predator Evolution Process (APEP) and use the following notation throughout the remainder of the section:
\begin{itemize}
\item $N_n = $ number of coexisting predators at time $n$.
\item $\alpha_j(n) = j^{th}$ largest $\alpha$ amongst all coexisting predators at time $n$ for $j \leq N_n$
\item $\alpha_{\min}(n) = \alpha_{N_n}(n)$
\item $\alpha_j(n) = \alpha_{\min}(n)$ if $j > N_n$.
\item $d_j(n) = \alpha_j(n) - \alpha_{\min}(n) =$ differences between predator consumption rates.
\item $\Delta_n = (d_1(n),d_2(n),\ldots)$
\end{itemize}
\n Setting $\alpha_j(n) = \alpha_{\min}(n)$ for $j > N_n$ is for convenience so that $d_j = 0$ for all $j\geq N_n$.

Substituting $\delta=1$, the condition (\ref{comain}) for coexistence of $\alpha_1,...,\alpha_N$ simplifies to
\beq
\sum_{j=1}^{N} \frac{\alpha_j}{\alpha_N} (\alpha_j-\alpha_N) < r - \frac{\beta}{\alpha_N}
\label{fixdelcond}
\eeq
Since $\alpha_j/\alpha_N>1$, this implies that all the differences $\alpha_j-\alpha_N$ must be $<r$ so we define $\mathcal{S} := [0,r]^{\mathbb{N}}$ and let $\|\cdot\|_{TV}$ denote the total variation norm on $\mathcal{M}_1(\mathcal{S})$ = space of probability measure on $\mathcal{S}$. We denote by $P^\alpha$ the law of the APEP started from an initial predator with trait $\alpha$.

\begin{theorem} \label{deltafixedmain}
$N_n$ is tight and $\alpha_{\min}(n) \rightarrow \infty$ a.s. as $n \rightarrow \infty$. In addition, there exists a measure $\pi_\ep$ on $\mathcal{S}$ and constant $a_\ep>0$ so that
$$
\| P^{\alpha(0)}(\Delta_n \in \cdot) - \pi_\ep(\cdot)\|_{TV} \rightarrow 0
$$
and $\alpha_{\min}(n)/n \rightarrow a_\ep > 0$ as $n \rightarrow \infty$.
\end{theorem}

This result is proved in Section \ref{deltafixedsec}. The reason for the difference from Theorem \ref{preymain} is that mutant types with traits similar to the resident type can always invade because predators only suffer density dependent killings from their own type. In the case of prey evolution, this is not the case since inter-and intra- specific competition equally affect the prey.

The key to the proof of Theorem \ref{deltafixedmain} is the observation that as $\alpha_N\to\infty$, the condition (\ref{fixdelcond}) becomes
$$
\sum_{j=1}^{N} (\alpha_j-\alpha_N) < r
$$
and we can show that the differences $\Delta_n$ are asymptotically a positive recurrent Harris chain with stationary distribution $\pi_\ep$. A coupling argument shows that the nonhomogeneous chain also converges to $\pi_\ep$. The linear growth of $\alpha_{min}$ then follows from a standard result on functionals of positive recurrent Markov chains.

Figure \ref{eps01alp3} illustrates the tightness of $N_n$ and linear growth of $\alpha_{\min}$. Figure \ref{verysmalleps} suggests that as the size of the perturbation $\ep\to 0$, the spacings between traits is $O(\ep)$, and the number of coexisting types is $O(1/\ep)$. We believe that if one converts the re-scaled spacings $\Delta_n/\ep$ into a measure by assigning each one mass $\ep$ then as $\ep\to 0$, the distribution of this measure under $\pi_\ep$ converges to a deterministic limit in which the density of particles is roughly, but not exactly, exponential, see Figure \ref{alphaden}.

\subsection{The Delta Predator Evolution Process } \label{fixedalphaintro}

The final evolution process we consider is the Delta Predator Evolution Process (DPEP). The DPEP is defined in continuous time and follows the same rules as the Predator EP except that all predators have fixed  $\alpha=1$ and we only allow for mutations in $\delta$. For convenience, we assume that $\ep=1$ and define $X_j(t) = -\log \delta_j(t)$. We also set $N_t$ = the number of coexisting predators at time $t$. Note that since $\delta_1(t) < \delta_2(t) < \cdots < \delta_{N_t}(t)$ by definition of the DPEP, we always have $X_1(t) > X_2(t) > \cdots >X_{N_t}(t)$ so that $X_{\max}(t) = X_1(t)$ and $X_{\min}(t) = X_{N_t}(t)$ give the traits of the most and least fit predators, respectively. Furthermore, \eqref{comain} implies that
\beq \label{contscond}
e^{-X_{N_t}(t)} \left( \beta + \sum_{j=1}^{N_t} 1-\exp(-[X_j(t)-X_{N_t}(t)]) \right) < r
\eeq
for all $t \geq 0$.

To get started in the analysis of the DPEP, our first step in Section \ref{alphafixedsec} is to prove a simple result which already shows that the behavior is much different from the APEP.

\begin{lemma} \label{alphafixedlem}
As $t \rightarrow \infty$, $N_t \to \infty$ a.s.
\end{lemma}

Our next result describes the asymptotic rates at which the smallest and largest predator traits increase. In what follows, we let $S_t$ be a random walk starting at $0$ that takes jumps at rate 1 uniform on $[-1,1]$. The theory of large deviations tells us that
$$
\Lambda(x) = \lim_{t\to\infty} \frac{1}{t} \log P( S_t > xt )
$$
exists and can be calculated in terms of the moment generating function of $S_t$.

\begin{theorem} \label{alphafixedmain}
$X_{\max}(t)/t \rightarrow a$ and $X_{\min}(t)/ t \rightarrow b$ a.s.~as $t \rightarrow \infty$ where $a \approx 0.9053$ and $b \approx 0.5667$ satisfy the equations
\beq \label{ab}
\Lambda(a) = -1, \qquad \Lambda(b) = -1 +b.
\eeq
Furthermore, we have $\liminf (1/t) \log N_t \geq b$ a.s.
\end{theorem}

We will prove Theorem \ref{alphafixedmain} in Section \ref{alphafixedsec}. To explain why it is true, it is convenient to adopt the perspective that the different predator types correspond to particles and their traits correspond to positions on the real line. Let $Z_t$ be the branching random walk in which particles give birth at rate 1 and their offspring are displaced by an amount uniform on $[-1,1]$. A result of Biggins \cite{JB1} implies that $r_t$, the position of right-most particle at time $t$ in the branching random walk, has $r_t/t \to a$ and
$$
\frac{1}{t} \log Z_t([xt,\infty)) \to 1 + \Lambda(x)
$$
for $0 \le x < a$ so $(1/t) \log Z_t([bt,\infty)) \to b$. Since we can construct $Z_t$ in such a way that all particles in $X(t)$ are in $Z_t$, we must have $\limsup X_{\max}(t)/t \leq a$ a.s.  The definition of $b$ and an argument by contradiction using (\ref{contscond}) gives the upper bound $\limsup_{t\to\infty} X_{\min}(t)/t \le b$ for the speed of the left-most particle.

To bound $\liminf_{t\to\infty} X_{\max}(t)/t$, we consider the following branching-selection particle system: at any time $t$, we have $M$ particles with positions $Y_1^M(t) > \cdots > Y_M^M(t)$, all giving birth at rate one. Whenever a new particle is born, we reorder and delete the leftmost particle. Using techniques from the study of the APEP in Section \ref{deltafixedsec}, we could show that $Y_1^M(t)/t \to a_M$, but instead we complete the proof of the first result by showing
\beq \label{TOYlim}
\lim_{M\to\infty} \liminf_{t\to\infty} Y^M_1(t)/t = a.
\eeq
Nina Gantert has pointed out to us that Berard and Gouere (2008) have recently proved
$$
a - a_M \sim C(\log(M))^{-2}
$$
for a related discrete time model in which all $M$ particle split into two and only the right-most $M$ are kept. This confirms a slow rate of convergence, which was predicted much earlier by Brunet and Derrida \cite{BD}, and which we observed in our numerical attempts to verify the limit in (\ref{TOYlim}), see Figure \ref{aMslow}.

To bound $\liminf_{t\to\infty} X_{\min}(t)/t$, we study the branching random walk with killing at $-K + \gamma t$. Our result given in Lemma \ref{killedBRW} is a cousin of a result of Kesten \cite{HK} for branching Brownian motion on $[0,\infty)$ where during its lifetime, each particle moves according to Brownian Motion with drift $\mu < 0$ and variance $\sigma^2$, all particles die at rate $c$ and give birth to a mean $m$ number of offspring upon death with particles killed when they hit 0. Kesten's result states that the system has positive probability of survival when $\mu < \mu_0 = (2\sigma^2c(m-1))^{1/2}$ (Theorem 1.1, (1.6)), and in this supercritical case, if we start with one particle at $x$, then for every interval $I$, $Z_t(I)/E_x Z_t(I) \to W$ a.s for some finite random variables $W$ (Theorem 1.1 (1.5)). However, Kesten's efforts are concentrated on the exotic behavior in the critical case $\mu = \mu_0$, and he says ``so far we have only an ugly and complicated proof of the growth results in the supercritical case, and we shall therefore not prove Theorem 1.1.'' In section 5, we show that using ideas of Biggins \cite{JB1} it is easy to prove results for $(1/t) \log Z_t([ct,\infty))$.

The result $\limsup_{t\to\infty} X_{\min}(t)/t \le b$ implies that if $T$ is large and we start the branching random walk with one particle at $X_{max}(T)$ at time $T$ then all of the particles in the branching random walk with killing at $(b+\ep)t$ are present in the $X_i(t)$. If $X_{\min}(t)$ is too far to the left then we would contradict (\ref{contscond}). The last part of the proof suggests that most particles are near $X_{\min}(t)$. Simulations (see Figure \ref{Xdens}) further suggest that:

\mn
{\bf Conjecture.} {\it
If we put mass $\exp(-X_{\min}(t))$ at $X_i(t)-X_{\min}(t)$ then this measure converges to a deterministic limit, which again is roughly but not exactly exponential.}

\mn
However, proving this seems to be a difficult problem. Recently, Durrett and Remenik \cite{DR} have proved convergence of the toy model to the solution of a free boundary problem as $M \to \infty$.

The final conclusion $\liminf_{t\to\infty} (\log N_t)/t \ge b$ follows from the result for $X_{\min}(t)$ and the proof of Lemma \ref{alphafixedlem}. Since the result comes from replacing (\ref{contscond}) by
$e^{-X_N(t)}(\beta+N) < r$, it seems unlikely that $b$ is the right constant, but finding the right constant would require proving the conjecture.

The proof of $\liminf_{t\to\infty} X_{\max}(t)/t \ge a$ also leads to the following result regarding the limiting behavior of the most fit predator in the APEP as the mutation radius $\ep \to 0$.

\begin{corollary} \label{fixdeltspeed}
If we run the APEP in continuous time and let
$$a^*_\ep = \lim_{t\to\infty} \alpha_{\max}^\ep(t)/t,$$
then $\lim_{\ep\to 0} a^*_\ep = a$.
\end{corollary}

\section{Prey Evolution} \label{preysec}

In this section we will prove Theorem \ref{preymain}. We first establish Proposition \ref{preyODEmain} as a consequence of Lemmas \ref{2prey1pred} and \ref{3prey1pred} below. Note that these results only cover $M \leq 3$, but the proof of Lemma \ref{3prey1pred} can easily be extended to show that coexistence for $M \geq 4$ is not possible. Our results rely heavily on the notion of invadability (see Durrett \cite{RD2}) and results on Lotka-Volterra systems (see Chapter 13 and 15 of Hofbauer and Sigmund \cite{HS}) which we shall quote as needed. We shall also make use of the notation introduced in Section \ref{preysecintro} and assume throughout this section that $N=1$ in \eqref{LV}. Therefore, we have a single predator, whose density we shall denote by $v$, with fixed death rate $\delta >0$.

\subsection{Prey ODE Results}

Let $y_1 = (\alpha_1,\beta_1) \in \R_+^2$. To determine when the predator and a prey with trait $y_1$ can coexist, we note that if $\beta_1>1$ then in the absence of predators the prey reach an equilibrium density
\beq
\sigma_1^0(y_1) = (\beta_1-1)/\beta_1.
\label{preyeq}
\eeq
If the prey are in equilibrium then the predators can increase when $v$ is small if
\beq
\alpha_1 \sigma_1^0(y_1) - \delta >0.
\label{predinv}
\eeq
Using the formula for $\sigma_1^0(y_1)$, we see that this holds if and only if $\alpha_1>\delta$ and
\beq
\beta_1 > \frac{\alpha_1}{\alpha_1-\delta} > 1.
\label{viable}
\eeq
We call this set of $(\alpha_1,\beta_1)$ the {\it viable region} for prey and label it $\mathcal{V}$. See Figure \ref{preyinvadepic} for an example.

Algebra shows that when (\ref{viable}) occurs, there is a predator-prey equilibrium $\sigma^1(y_1)$ with coordinates
\beqx
\sigma_1^{1}(y_1)= \frac{ (\beta_1-1) + \alpha_1\delta }{ \beta_1  + \alpha^2_1 },
\qquad
\sigma_2^{1}(y_1) = \frac{(\beta_1-1)\alpha_1 - \beta_1\delta}{ \beta_1  + \alpha^2_1}.
%\label{twospecieseq}
\eeqx
A second prey type with trait $y_2=(\alpha_2,\beta_2)$ can invade the first prey and the predator in equilibrium when
\begin{align*}
\beta_2 (1-\sigma_1^{1}(y_1)) - 1 - \alpha_2 \sigma_2^{1}(y_1) &> 0.
%\label{2inv13int}
\end{align*}
By interchanging the roles of $y_1$ and $y_2$ we get the condition for the first prey to invade the second prey
and predator in equilibrium. If both prey traits are viable and the two invadability conditions hold, then it is shown in Section 7.1 of \cite{RD2} that there is coexistence in the ODE, i.e., the three densities stay bounded away from 0. That the densities actually converge to a positive equilibrium in this case is the result of Lemma \ref{2prey1pred} below.

Following \cite{RD2}, we use $\succ$ for ``invades" (prey $j$ can invade prey $1,...,j-1$ in equilibrium if its density will increase whenever $1,...,j-1$ are in equilibrium with the predator and a small initial density of $j$'s is introduced). Using the new notation and defining
$$
F(y_1,y_2) = \beta_2 (1-\sigma_1^1(y_1)) - 1 - \alpha_2 \sigma_2^1(y_1),
$$
we have $2 \succ 1$ if and only if
$$
y_2 \in \{y: F(y_1,y) > 0\}=:\mathcal{L}_{y_1}
$$
and $1 \succ 2$ if and only if
$$
y_2 \in \{y: F(y,y_1) > 0\} =:\mathcal{U}_{y_1}.
$$
We call the boundary curves $L_{y_1} \equiv \{y: F(y_1,y) = 0\} $ and $U_{y_1} \equiv \{y: F(y,y_1) = 0\}$ the invadability curves and note that we have the formulas
\begin{align*}
(\alpha,\beta) \in U_{y_1} &\Leftrightarrow \beta = g(y_1,\alpha) \\
(\alpha,\beta) \in L_{y_1} &\Leftrightarrow \beta = h(y_1,\alpha)
\end{align*}
where
$$
g(y_1,\alpha) = \frac{(\beta_1-1)\alpha^2 +(\alpha_1-\beta_1\delta)\alpha + \beta_1}{1 + \alpha_1(\alpha-\delta)}
$$
and
$$
h(y_1,\alpha) = \frac{\alpha \sigma_2^1(y_1)+1}{1-\sigma_1^1(y_1)}
$$
are well defined provided $y_1,(\alpha,\beta) \in \mathcal{V}$. Calculus shows that the curve $U_{y_1}$ is tangent to the curve $L_{y_1}$ at $y_1$ and we let $\mathcal{N}(y_1)$ denote the corresponding unit normal vector:
\beq
\mathcal{N}(y_1) = c (-\sigma_2^1(y_1),1-\sigma_1^1(y_1))
\label{vectorN}
\eeq
where $c$ is chosen to make the length one. The situation is depicted in Figure \ref{preyinvadepic}. Lemma \ref{2prey1pred} describes the set of all possible ecological outcomes based on this splitting of type space.

\begin{lemma} \label{2prey1pred}
Let $y_1,y_2 \in \mathcal{V}$ with $\beta_1 \neq \beta_2$ and suppose that $(u_1(0),u_2(0),v(0)) \in \Gamma_{2,1}^+ $. Then one of the following must be true
\begin{description}
\item[(a)] $y_2 \in \mathcal{L}_{y_1} \cap \mathcal{U}_{y_1}$ in which case the solution to \eqref{LV} converges to a unique positive equilibrium $\sigma^2(y_1,y_2)= (\sigma_1^2(y_1,y_2), \sigma_2^2(y_1,y_2), \sigma_3^2(y_1,y_2)) \in \Gamma_{2,1}^+$.
\item[(b)]  $y_2 \in \mathcal{L}_{y_1}$, but $y_2 \notin \mathcal{U}_{y_1}$ in which case the solution to \eqref{LV} converges to the equilibrium $(0,\sigma_1^1(y_2),\sigma_2^1(y_2))$.
\item[(c)] $y_2 \in \mathcal{U}_{y_1}$, but $y_2 \notin \mathcal{L}_{y_1}$ in which case the solution to \eqref{LV} converges to the equilibrium $(\sigma_1^1(y_1),0,\sigma_2^1(y_1))$.
\end{description}
\end{lemma}

Since inserting a mutant with the same birth rate as the resident is zero, the condition $\beta_1 \neq \beta_2$ does not impose any additional restrictions and saves us the headache of dealing with a scenario in which we have an infinite number of equilibria.

Before beginning the proof, we note that setting $u_3 = v$, we can rewrite \eqref{LV} in Lotka-Volterra form as
$$
\frac{du_i}{dt} = u_i(r_i + (Au)_i)
$$
where $r_i = \beta_i-1$, $i=1,2$, $r_3 = -\delta$, and
$$
A = \left[\begin{array}{ccc} -\beta_1 & -\beta_1 & -\alpha_1\\
-\beta_2 & -\beta_2  & -\alpha_{2}  \\
\alpha_{1} & \alpha_{2} & -1
\end{array}\right].
$$
If $f_i(u) = du_i/dt$, then the components of the Jacobian are given by
\beq \label{jacob}
\frac{df_i(u)}{du_j} = \delta_{ij}(r_i+(Au)_i) + u_iA_{ij}.
\eeq
Let $J_p$ denote the value of the Jacobian matrix $J = (df_i/du_j)$ evaluated at $p$. Following \cite{HS} (see pages 155, 159), we shall say that an equilibrium point $p$ for \eqref{LV} is \emph{regular} if $\det(J_p) \neq 0$ and \emph{saturated} if $r_i+(Ap)_i\leq 0$ for all $i$. Note that if we have an equilibrium $p \in \Gamma_{2,1}^+$, then $p$ is trivially saturated since in this case, $r_i +(Ap)_i = 0$ for all $i$. More generally, an equilibrium point $\sigma$ with $\sigma_i = 0$ for all $i \in I$ is saturated if it cannot be invaded by types $i \in I$. In particular,  $(\sigma_1^1(y_1),0,\sigma_2^1(y_1))$ and $(0,\sigma_1^1(y_2),\sigma_2^1(y_2))$ are saturated exactly when $y_2 \notin \mathcal{L}_{y_1}$ and $y_2 \notin \mathcal{U}_{y_1}$, respectively. The assumptions $y_1,y_2 \in \mathcal{V}$ and $\beta_1,\beta_2 > 1$ imply that $\sigma^0(y_1)$, $\sigma^0(y_2)$, and the origin are never saturated. It is easy to see that since $\beta_1 \neq \beta_2$, there can be no other possible equilibria $\sigma \notin \Gamma_{2,1}^+$.

Let
$$
ind(p) =  sign(\det(J_p))
$$
denote the \emph{index} of a regular equilibrium $p$. The index theorem for Lotka-Volterra equations (13.4.4 in \cite{HS}) tells us that if all saturated equilibria $p$ are regular, we must always have
\beq \label{sat}
\sum_{p: p \, saturated} ind(p) = (-1)^3 = -1.
\eeq
The key to the proof will be showing that \eqref{LV} has a unique saturated fixed point in all three cases (a) - (c). However, since it is not always true that a unique saturated fixed point is globally attracting (see pg. 195 in \cite{HS}), we need to work a little bit harder to get the result. To ease notation, we shall let $F_i$ denote the face in $\Gamma_{2,1}$ on which $u_i=0$ and $E_{i,j}$ denote the edge where $u_i=u_j=0$.

%We also recall that $A$ is VL-stable if there exists a diagonal matrix $D$ so that
%\beqx
%\sum_{i,j} d_i A_{ij} y_i y_j < 0
%\eeqx
%for all $y \in R^3$. If $A$ is VL-stable and there is a positive equilibrium for the associated Lotka-Volterra system, then the positive equilibrium is globally attracting by Theorem 15.3.1 in Hofbauer and Sigmund \cite{HS}

\begin{proof}
Suppose first that we are in case {\bf (a)} so that $(\sigma_1^1(y_1),0,\sigma_2^1(y_1))$ and $(0,\sigma_1^1(y_2),\sigma_2^1(y_2))$ are not saturated. Then Theorem 7.1 in \cite{RD2} implies we have coexistence and hence by Theorems 13.3.1 and 13.5.2 in \cite{HS}, \eqref{LV} has a unique regular equilibrium $\sigma^2 \in \Gamma_{2,1}^+$. To show that it is globally attracting, we will show that all eigenvalues of $J_{\sigma^2}$ have negative real parts. The conclusion that $\sigma^2$ is globally attracting on $\Gamma_{2,1}^+$ then follows by Theorem 15.3.1 in \cite{HS} (A is  Volterra-Lyapunov stable with $d_i=\sigma^2_i$). The Routh-Hurwitz (R-H) conditions (see pages 702-703 of Murray \cite{JM} for the version used here or \cite{AD} for an elementary proof) tell us that all eigenvalues of $J_{\sigma^2}$ will have negative real parts if (1) trace$(J_{\sigma^2})  < 0$, (2) $\det(J_{\sigma^2}) < 0$, and (3)
\beqx \label{rhcond}
\det(J_{\sigma^2}) > \text{trace}(J_{\sigma^2}) \Sigma_2
\eeqx
where $\Sigma_2$ is the sum of the $2\times 2$ principal minors of $J_{\sigma^2}$. But since $r+A\sigma^2 = 0$, substituting $\sigma^2$ into \eqref{jacob} yields
$$
J_{\sigma^2}  =  \left[\begin{array}{ccc} -\sigma^2_1 \beta_1 & -\sigma^2_1 \beta_1 & -\sigma^2_1 \alpha_1\\
-\sigma^2_2 \beta_2 & -\sigma^2_2 \beta_2  & -\sigma^2_2 \alpha_{2}  \\
 \sigma^2_3 \alpha_{1} &  \sigma^2_3 \alpha_{2} & -\sigma^2_3
\end{array}\right].
$$
so that the first R-H condition is obvious and the third follows from a simple algebraic calculation. The second condition follows from \eqref{sat} since $\sigma^2$ is the unique saturated equilibrium point for \eqref{LV} and is regular.

Suppose now we are in case {\bf (b)} so that $(0,\sigma_1^1(y_2),\sigma_2^1(y_2))$ is saturated, but $(\sigma_1^1(y_1),0,\sigma_2^1(y_1))$ is not saturated. If we let $\sigma = (0,\sigma_1,\sigma_2) = (0,\sigma_1^1(y_2),\sigma_2^1(y_2))$, then
$$
J_\sigma = \left[\begin{array}{ccc}  r_1+(A\sigma)_1 &0 & 0  \\-\sigma_1 \beta_2 & -\sigma_1 \beta_2 & -\sigma_1 \alpha_2\\
\sigma_2 \alpha_{1} &  \sigma_2 \alpha_{2} & -\sigma_2 \gamma
\end{array}\right].
$$
The assumptions in {\bf (b)} imply that $2 \succ 1,3$ and $1\nsucc 2,3$ so we must have $r_1 + (A\sigma)_1 < 0$. Therefore,
$$\det(J_\sigma)  = (r_1+(A\sigma)_1)(\beta_2\gamma + \alpha_2^2)\sigma_1\sigma_2 < 0$$
implying that $\sigma$ is regular. If coexistence was possible, then as in the proof of {\bf (a)}, we would have a regular equilibrium $\rho \in \Gamma_{2,1}^+$ and \eqref{LV} would have exactly two regular, saturated equilibria, violating \eqref{sat}. Therefore we know that $ u_i(t) \to 0$ as $t \to \infty$ for some $i=1,2,3$. But since we have the invadability conditions, $1 \succ 0$, $3 \succ 1$, $3 \succ 2$, $2 \succ 1,3$, and $1\nsucc 2,3$, the proof of Theorem 7.1 in \cite{RD2} tells us that there exists a repelling function for the set
$$F \equiv F_3 \cup E_{1,2} \cup F_2$$
and therefore we know that \eqref{LV} must leave $\Gamma_{2,1}^+$ through $F_1\backslash F $ on which $\sigma$ is globally attracting (Lemma 5.0 in \cite{RD2}). The proof of {\bf(c)} is identical after interchanging the roles of $y_1$ and $y_2$.
\end{proof}

Our next result describes the possible outcomes of adding a new species when two are already coexisting.

\begin{lemma} \label{3prey1pred}
Let $y_1, y_2,y_3 \in \mathcal{V}$, $y_2 \in \mathcal{L}_{y_1} \cap \mathcal{U}_{y_1}$, $\beta_1 \neq \beta_2$, and
$$(u_1(0),u_2(0),u_3(0),v(0)) \in \Gamma_{3,1}^+.$$
Then one of the following must be true
\begin{description}
\item[(a)] $y_3 \in \mathcal{L}_{y_1} \cap \mathcal{L}_{y_2}$ and $y_3 \notin \mathcal{U}_{y_1} \cup \mathcal{U}_{y_2}$ in which case the solution to \eqref{LV} converges to $(0,0,\sigma_1^1(y_3),\sigma_2^1(y_3))$.
\item[(b)]  $y_3 \in \mathcal{U}_{y_1} \cap \mathcal{U}_{y_2}$ and $y_3 \notin \mathcal{L}_{y_1} \cup \mathcal{L}_{y_2}$ in which case the solution to \eqref{LV} converges to $(\sigma_1^2(y_1,y_2),\sigma_2^2(y_1,y_2),0,\sigma_3^2(y_1,y_2))$.
\item[(c)] Neither {\bf (a)} nor {\bf (b)} is satisfied in which case either $u_1(t) \to 0$ or $u_2(t) \to 0$.
\end{description}

\end{lemma}

\begin{proof}
It is easy to check that if we have $M=3$ in \eqref{LV}, relabel the predator's density as $u_4$ and write \eqref{LV} in Lotka-Volterra form, then $\det(A) = 0$ so that there can be no coexistence of the three prey types by Theorem 13.5.2 in \cite{HS}. Therefore, at least one of the types dies out. Which one(s) can be sorted out using the same idea as the proof of Lemma \ref{2prey1pred}.
\end{proof}

To see why we do not need to be concerned with case {\bf (c)}, we note that because of the tangency of $U_{y_1}$ and $L_{y_1}$, the chance of inserting a mutant which can coexist with a current resident prey is always of order
$$
\ep^{-2} \int_0^\ep x^2 dx = O(\ep).
$$
But then the probability of inserting a third prey which can coexist with two resident prey is also $O(\ep)$ and therefore, we must wait $O(1/\ep^2)$ time steps until the first time we encounter the situation in Lemma \ref{3prey1pred}, {\bf (c)}. Of course, we will still get convergence to an equilibrium that can be determined as in Lemma \ref{2prey1pred} provided $\beta_3 \neq \beta_1, \beta_2.$

%We conclude this section with a summary of the possible ecological outcomes that can occur after each mutation.
%\begin{description}
%\item[(i)] If $y_1(n) > 0$ and $y_2(n)=0$, then to determine $\mathbf{y}(n+1)$, we check if $y_{new} \in \mathcal{L}_{y_1(n)} \cap \mathcal{U}_{y_1(n)}$. If so, then we take $y_1(n+1) = y_1(n)$ and $y_2(n+1) = y_{new}$. If not, then we take $y_1(n+1)= y_1(n+1)$ if $y_{new} \in \mathcal{U}_{y_1(n)}$, $y_1(n+1) = y_{new}$ if $y_{new} \in \mathcal{L}_{y_1(n)}$, and $y_2(n+1)=0$.
%\item[(ii)] If $y_1(n), y_2(n) > 0$, then we have three possibilities. If $y_{new} \in \mathcal{U}_{y_1(n)} \cap \mathcal{U}_{y_2(n)}$, but $y_{new} \notin \mathcal{L}_{y_1(n)} \cup \mathcal{L}_{y_2(n)}$ then we encounter the situation in Lemma \ref{3prey1pred}, {\bf (b)} so we let $\mathbf{y}(n+1) = \mathbf{y}(n)$. If $y_3 \in \mathcal{L}_{y_1} \cap \mathcal{L}_{y_2}$, but $y_3 \notin \mathcal{U}_{y_1} \cup \mathbf{U}_{y_2}$, then we encounter the situation in Lemma \ref{3prey1pred}, {\bf (a)} and we let $y_1(n+1) = y_{new}$ and $y_2(n+1) = 0$. If $y_{new}$ is somewhere else, then we encounter the situation in  Lemma \ref{3prey1pred}, {\bf (c)} and we know either $y_1$ or $y_2$ will die out. We let $n_\ep$ denote the first time this occurs and will show below that $\ep n_\ep \to 0$ as $\ep \to 0$.
%\end{description}

\subsection{Proof of Theorem \ref{preymain}}

Throughout thus section, we shall use $C=C_T$ to denote a positive constant which depends on $T$ and may change from line to line. Write
$$
F^\ep(y_1) = \frac{1}{2\pi \ep^2} \int_{-\ep}^\ep f(y_1,x)dx
$$
where $f(y_1,\alpha) = g(y_1,\alpha)- h(y_1,\alpha) \geq 0.$ We also define
\beqx
p^\ep(t) =\begin{cases} F^\ep(Y_1^\ep(t)) & \text{ if } Y_2^\ep (t) = 0 \\
\max(F^\ep(Y_1^\ep (t)),F^\ep(Y_2^\ep(t))) & \text{ if } Y_2^\ep(t) \neq 0.
\end{cases}
\eeqx
For $Y_1^\ep(t), Y_2^\ep(t)\in \mathcal{V}$, $p^\ep$ gives the probability of inserting a new type that can coexist when only one resident type is present and an upper bound on the probability a new type can coexist with one of the resident types when two resident types are present. Note that the tangency of $g(y_1,\cdot), h(y_1,\cdot)$ and Taylor's theorem imply that
\beq \label{fepslim}
\ep^{-1} F^\ep(y_1) \to c \phi(y_1)
\eeq
for any $y_1 \in \mathcal{V}$ and some constant $c >0$ where
$$
\phi(y_1) = \frac{\partial^2(g-h)}{\partial \alpha^2}(y_1,\alpha_1)
$$
is a continuous function of $y_1$. Therefore, if $Y_1^\ep(t), Y_2^\ep(t) \in K$, $K$ bounded, then there exists $C_K > 0$ so that $p^\ep(t) \leq C_K \ep$.

Choose an open, bounded set $K_1 \subset \mathcal{V}$ and a compact set $K_2 \subset K_1$ so that $y_1(t) \in K_2$ for all $t \leq T$. The existence of $K_1$ is guaranteed since on the boundary of the viable region, $\beta=\alpha/(\alpha-\delta)$ so that the slope of $U_{y_1}$ at $y_1$ is 0 implying that $\mathcal{N}(y_1)$ points straight up, and it is impossible for $y_1(t)$ to leave the viable region. Let $\rho >0$ be small enough so that $K_2 + \rho \subset K_1$ and define
$$\tau = \inf\{t: \mathbf{Y}^\ep(t) \notin (K_2+\rho) \times (K_2+\rho)\}$$
as the first time $Y_1^\ep(t)$ or $Y_2^\ep(t)$ leaves $K_2 + \rho$. Then $p^\ep(t\wedge \tau) \leq C \ep$, $\forall t \leq T$. Since mutations occur at rate $1/\ep$, it follows that the expected number of times before $T \wedge \tau$ that there is one prey type and an inserted type coexists is $\leq C$. If two prey types coexist, Lemma \ref{3prey1pred} implies that with probability $\geq 1-C \ep$, the next time a new type is inserted in $\mathcal{L}_{y_1} \cap \mathcal{L}_{y_2}$, it will replace the two coexisting types. Therefore, if $\ep$ is small, the amount of time during which two types coexist is approximately exponential with mean $2\ep$ so that
\beq \label{y2neg}
|\{t\leq T: Y_2^\ep(t\wedge \tau) > 0\}| \to 0
\eeq
a.s.~as $\ep \to 0$ and hence, we can ignore these isolated episodes when studying the evolution of $Y_1^\ep(t\wedge \tau)$.

%Since mutations always occur to $y_{new} \in B_\ep(Y_i^\ep)$, for some $i=1,2$, we have
%$$
%d(Y_i^\ep(t),Y_1^\ep(0)) \leq 2 \ep \times \text{ Number of Mutations in } [0,T].
%$$
%for $i=1,2$ and since the expected number of mutations in $[0,T]$ is $T/\ep$, there exists a bounded set $K = K_T \subset \mathbb{R}^2$ so that $P(\tau_K^\ep \leq T) \to 0$ as $\ep \to 0$ where
%$$\tau_K^\ep = \inf\{t: \mathbf{Y}^\ep(t) \notin K \times K\}$$
%is the first time $Y_1^\ep(t)$ or $Y_2^\ep(t)$ leaves $K$.

When there is no coexistence, mutations in the direction of $L_{y_1}$ leave the resident type unchanged and mutations in the direction of ${\cal N}(y_1)$ replace the resident so that the infinitesimal mean of $Y_1^\ep(t\wedge \tau)$ is given by
\beq \label{infmean}
b(y_1) = \frac{2}{3\pi} {\cal N}(y_1)
\eeq
where the $2/3\pi$ comes from the fact that if we choose a point at random from the upper half of the ball of radius 1 in the $(\alpha,\beta)$ plane, then the $\beta$ component has density $(4/\pi) \sqrt{1-\beta^2}$ and hence mean
$$
\frac{4}{\pi} \int_0^1 \beta \sqrt{1-\beta^2} \,dy  =\frac{4}{3\pi}.
$$
(\ref{infmean}) then follows by noting that choices from the half of the ball above $L_{y_1}$ occur with probability 1/2. It is clear from the scaling that the entries in the infinitesimal covariance are of order $\ep$ and therefore, the infinitesimal mean and covariance of $Y_1^\ep(\cdot \wedge \tau)$ converge to $b(y_1)$ and $a(y_1) = 0$ respectively. Since $b$ is Lipschitz continuous, the martingale problem for $(a,b)$ is well posed so that convergence of $Y_1^\ep(\cdot \wedge \tau)$ to $y_1$  follows from Theorem 7.4.1 in Ethier and Kurtz \cite{EK}. But then we can choose $\rho$ small enough so that $P(\tau \leq T) \to 0$ and we obtain \eqref{FAD}.

It remains to prove that
$$N_t^\ep = |\{s \leq t : Y_2^\ep(s-)=0, \, Y_2^\ep(s) \neq 0\}|$$
converges to a nonhomogeneous Poisson process. Since $F^\ep(Y_1^\ep(t))$ gives the jump probabilities for $N_t^\ep$ when $Y_2^\ep(t)=0$, the compensator for $N_t^\ep$ is given by
$$
A_t^\ep = \int_0^t \mathbf{1}_{\{Y_2^\ep(s)=0\}}\ep^{-1}F^\ep(Y_1^\ep(s)) ds.
$$
\eqref{fepslim}, \eqref{y2neg}, and \eqref{FAD} then imply that
\beq \label{poislim}
A_t^\ep \to m(t) \equiv \int_0^t  c \phi(y_1(s))ds.
\eeq
$m(t)$ is continuous and deterministic so we conclude from Theorem 1 in Brown \cite{TB} that $N^\ep \Rightarrow N$ where $N$ is a nonhomogeneous Poisson Process with mean function $m(t)$.

\eopt

\section{Multiple Predator ODE facts} \label{predatorsec}

%$$
%\frac{dv_i}{dt} = v_i(r_i + (Av)_i)
%$$
%for $i=1,2,...,N+1$ where $r_i = -\delta_i$, $A_{ii}=-1$, and $A_{i,N+1}= \alpha_i = -A_{N+1,i}$ for $i=1,...,N$, $A_{N+1,N+1} = -\beta$, and $r_{N+1} = \beta-1$.

The goal of this section is the derivation of Lemmas \ref{translaw} and \ref{lyap} which together imply Proposition \ref{odemain}. The first result gives the algebraic condition for existence of positive equilibrium densities and the second proves convergence to equilibrium. See Section \ref{predatorsecintro} for relevant notation.

In the absence of predators, the prey have equilibrium density $\sigma^0 = r/\beta$ where $r = \beta-1 > 0$ by assumption. Suppose we wish to find a positive equilibrium $\sigma^k = (\sigma_0^k,\sigma_1^k,\dotso,\sigma_k^k)$ on the face
$$
\Gamma_{1,k} = \{v \in \Gamma_{1,N} :  v_{k+1}=\cdots v_N = 0\}.
$$
Then, solving the equations $\alpha_j \sigma_0^k - \delta_j - \sigma_j^k = 0$ for $\sigma_j^k$, $j=1,...,k$ we obtain
\beqx
\sigma_j^k = \sigma_j^k(x_1,\dotso,x_k) =  \alpha_j \sigma_0^k - \delta_j
\eeqx
and substituting these expressions into the equation $r - \beta \sigma_0^k - \sum_{j=1}^k \alpha_j \sigma_j^k = 0$ yields
$$
r - \beta \sigma_0^k = \sum_{j=1}^k \alpha_j^2 \sigma_0^k - \sum_{j=1}^k \alpha_j \delta_j
$$
We can conclude that
\beqx
\sigma_0^k = \sigma_0^k(x_1,\dotso,x_k) = \frac{r + \sum_{i=1}^k \alpha_i\delta_i}{\beta + \sum_{i=1}^k \alpha_i^2} > 0
\eeqx
for all $k \geq 0$. To determine when $\sigma_j^k> 0$, $1 \leq j \leq k$, write $S_k = \sum_{i=1}^k \alpha_i^2$ and
$$
\alpha_j \sigma_0^k
= \frac{\alpha_j r + \alpha_j^2\delta_j + \alpha_j \sum_{i\neq j} \alpha_i\delta_i}{\beta + S_k}.
$$
Adding $\delta_j - \delta_j(\beta+S_k)/(\beta+S_k)$, the above
\begin{align*}
& = \delta_j + \frac{\alpha_j r -\beta\delta_j + \sum_{i\neq j} (\alpha_j \alpha_i\delta_i - \alpha_i^2\delta_j)}{\beta + S_k}\\
& = \delta_j + \frac{(\beta + \sum_{i\neq j} \alpha_i^2)(\alpha_j \sigma_0^{k-1}(x_1,...,x_{j-1},x_{j+1},...,x_k) - \delta_j)}{\beta + S_k}
\end{align*}
since $(\beta + \sum_{i\neq j} \alpha_i^2)\sigma_0^{k-1} = r + \sum_{i \neq j} \alpha_i\delta_i$.
From this it follows that $\sigma_j^k$ will be positive if and only if
\beq \label{noninvade}
\sigma_0^{k-1}(x_1,...,x_{j-1},x_{j+1},...,x_k) > \ell_j.
\eeq
where $\ell_j = \delta_j/\alpha_j$ is the characteristic ratio for predator $x_j$.

\begin{lemma} \label{translaw}
Suppose $x_1,...,x_k$ are ordered by increasing characteristic ratios. Then $\sigma_k^k> 0$ if and only if
\beq \label{master}
 \beta \ell_k + \sum_{j=1}^k \alpha_j^2 (\ell_k - \ell_j) < r
\eeq
and if \eqref{master} is satisfied, then $\sigma_j^k >0$ for all $j \leq k$.
\end{lemma}

\begin{proof} \eqref{noninvade} implies that $\sigma_k^k > 0$ if and only if
\beq \label{temp}
\ell_k < \sigma_0^{k-1}(x_1,\dotso,x_{k-1}) = \frac{ r + \sum_{j=1}^{k-1} \alpha_j^2 \ell_j}{\beta + \sum_{j=1}^{k-1} \alpha_j^2}
\eeq
where we have used the definition of $\ell_j = \delta_j/\alpha_j$ on the right. Multiplying both sides by the denominator of the right and then rearranging terms, we obtain \eqref{master} (since $\ell_k-\ell_k = 0$). Now suppose that \eqref{master} holds. Then since $\ell_k > \ell_j$, for all $j=1,...,k-1$, \eqref{master} also holds if we replace $\ell_k$ by $\ell_j$, $j < k$ on the left and reversing the algebra used to derive \eqref{master} from \eqref{temp} shows that this is equivalent to $\sigma_j^k > 0$, proving the result.
\end{proof}

\begin{lemma} \label{lyap}
Suppose we have a collection of predators $x_1,...,x_N$ ordered by increasing $\ell$'s and let $k \leq N$ be the largest integer for which \eqref{master} is satisfied (with the convention that $k=0$ if \eqref{master} fails for all $k \leq N$). Then $\sigma = (\sigma_0^k,\sigma_1^k,\ldots,\sigma_k^k, 0,\dotso,0)$ is a globally attracting fixed point on $\Gamma_{1,N}^+$ with Lyapunov function
\begin{equation*}
V(u,v_1,...,v_N) = u - \sigma_0^k \log u + \sum_{i=1}^k (v_i - \sigma_i^k \log v_i) + \sum_{i=k+1}^N v_i.
\end{equation*}
\end{lemma}

\begin{proof} Differentiating $V$ yields
\begin{eqnarray*}
  \frac{dV}{dt} &=& (u - \sigma_0^k)(r - \beta u - \sum_{i=1}^k \alpha_iv_i - \sum_{i=k+1}^N \alpha_iv_i) \\
  & & + \sum_{i=1}^k (v_i - \sigma_i^k)( -\delta_i -v_i +  \alpha_i u) +  \sum_{i=k+1}^N v_i( -\delta_i -v_i +  \alpha_i u) \\
  &=& - \beta (u - \sigma_0^k)^2 - \sum_{i=1}^k (v_i - \sigma_i^k)^2 - \sum_{i=k+1}^N v_i(\delta_i-\alpha_i \sigma_0^k)) - \sum_{i=k+1}^N v_i^2.
\end{eqnarray*}
If $(u,v_1,...,v_N) = \sigma$, this expression is 0 and otherwise it is $< 0$ since Lemma \ref{translaw} and \eqref{noninvade} imply that
$$
\ell_i \geq \ell_{k+1} > \sigma_0^k
$$
for all $i \geq k+1$ so that all terms on the left are negative.

\end{proof}

\section{Proof of Theorem \ref{deltafixedmain}} \label{deltafixedsec}

In this section, we prove Theorem \ref{deltafixedmain} for the APEP as defined in Section \ref{deltafixedintro} and use the notation defined there. We also define the Markov chain $Y_n = (\alpha_{\min}(n),\Delta_n)$.

If $\delta_j=1$ for all $j$, \eqref{master} with $k=N$ can be rewritten as:
\beq \label{master2}
\sum_{j=1}^N \frac{\alpha_j}{\alpha_N} (\alpha_j - \alpha_N) < r - \frac{\beta}{\alpha_N}.
\eeq
Our first step is to show

\begin{lemma} \label{tight}
The sequence $N_n$ is tight.
\end{lemma}

\begin{proof}
Define the sets $A^m = [0,r]^m \times \{0\}^\mathbb{N}$, for $m \in \mathbb{N}$. Then $\Delta_n \in A^m$ if and only if $N_n \leq m$. Let $M = M(r,\ep) = \lceil \frac{4r}{\ep}\rceil$ be the smallest integer $> 4r/\ep$ and suppose that $Y_n = y \in \mathbb{R}^{+} \times \mathcal{S}$. From \eqref{master2}, at most $M$ of the $\alpha_j(n)$'s can be $\ge \alpha_{\min}(n) + \ep/4$. With probability at least $1/4^M$, the next $M$ mutants will be inserted to the right of $\alpha_{\min}(n) + \ep/2$. But then none of the predators to the left $\alpha_{\min}(n) + \ep/4$ can be in the coexisting set at time $n+M$ because otherwise, by the definition of the APEP, any predator with $\alpha > \alpha_{\min}(n) + \ep/2$ would also be in the set, and since there are at least $M$ such predators,
$$
\sum_{j=1}^{\infty} d_j(n+M) > M (\ep/2 - \ep/4) > r
$$
contradicting \eqref{master2}. Therefore, we have the uniform lower bound
\begin{equation} \label{rec}
P(\Delta_{n+M} \in  A^{2M} | Y_n = y) \geq 4^{-M}
\end{equation}
which holds for all $y \in \mathbb{R}^{+} \times \mathcal{S}$. Since this bound is uniform in $y$,
tightness follows.
\end{proof}

\begin{lemma} \label{convtoMC}
As $n \rightarrow \infty$, the marginal transition probabilities for $\Delta_n$:
$$
p_\alpha(\Delta,\cdot) := P(\Delta_{n+1} \in \cdot| Y_n = (\alpha,\Delta))
$$
converge in total variation to the transition probabilities for a time homogeneous Markov Chain $X_n$ with transition probabilities $p(\Delta,\cdot)$.
\end{lemma}

\begin{proof}
Suppose that $Y_n = (\alpha,\Delta)$. Since $0 \le \alpha_j(n)-\alpha_{\min}(n)\le r$ for all $n \geq 1$, we can see that as $\alpha \to \infty$, (\ref{master2}) simplifies to
\beq
\sum_{j=1}^{N} d_j(n) < r.
\label{master2b}
\eeq
This implies that, in the limit, the differences evolve according to the following algorithm:
pick a species $1\le k\le N_n$ at random, insert a random mutation in $(d_k(n)-\ep,d_k(n)+\ep)$,
and then modify the algorithm in Theorem \ref{odemain} to use (\ref{master2b}) instead of
(\ref{master2}) with the rule that we shift the differences before calculating the sum if the new insertion is left of 0.
\end{proof}

Our next result concerns the limiting behavior of $X_n$. Writing $x$ instead of $\Delta$ for the vector of differences, we set
$$
p(x,A) = P(X_{n+1} \in A| X_n = x).
$$

\begin{lemma} \label{Harris}
$X_n$ is a positive recurrent, Harris Chain and hence, has a unique stationary distribution $\pi$.
\end{lemma}

\begin{proof}
Following the arguments in Athreya and Ney \cite{AN}, it suffices to show that there exists a ``regenerative'' set $A \subset \mathcal{S}$ satisfying:
\begin{description}
\item[(C1)] $P^x(\tau_A < \infty) = 1$ for all $x \in \mathcal{S}$ where $\tau_A$ is hitting time of $A$.
\item[(C2)] There exists a probability measure $\rho$ on $A$, $\lambda > 0$, and $\kappa \in \mathbb{N}$ so that $p^{\kappa}(x,B) \geq \lambda \rho(B)$ for all $x \in A$, $B \subset A$.
\end{description}

The same calculation that led to \eqref{rec} shows that $A^{2M}$ satisfies the condition in (C1), but (C2) may not hold for this set. We therefore define a set $G$ (for good) that will be reached from $A^{2M}$ with probability 1 and satisfies (C2). To this end, let
$$
\kappa = 1 + \sup\left\{k : \sum_{j=1}^k j = \frac{k(k+1)}{2} < 2r/\ep\right\}
$$
and choose $\eta$ small enough so that
\begin{equation} \label{etabnd}
\sum_{j=1}^{k} j(\ep/2 + \eta) < r.
\end{equation}
Let $G = \{ d_i-d_{i+1} \in (\ep/2,\ep/2+\eta)$ for $i< \kappa$ and $d_i=0$ for $i \ge \kappa\}$. In other words, $d \in G$ corresponds to $\kappa$ types coexisting with $\alpha$'s that have spacings between $\ep/2$ and $\ep/2 + \eta$ units apart.

The first step in showing that (C1) and (C2) hold for $A=G$ is to show that if $X_0 = x \in A^{2M}$, then we can get to $A$ in $\kappa$ steps by the following path: first, we choose $d_1$ (the predator with the largest values of $\alpha$) as our mutating predator at time $1$ (which happens with probability at least $(2M)^{-1}$) and then choose a mutant type $g_1$ in $(d_1+ \ep/2,d_1 + \ep/2 + \eta)$ (which happens with probability $\eta/(2\ep)$). The next time step, we choose $g_1$ as our mutating type (which happens with probability at least $(2M+1)^{-1}$) and then mutate to $g_2 \in (g_1+\ep/2,g_1+\ep/2 + \eta)$. If we continue for $\kappa$ steps, then each $g_j$, $1\le j\le \kappa$ will be at least as big as $d_1 + j\ep /2$ so that by \eqref{master2b}, no member of the coexisting set at time $0$ will remain at time $\kappa$. Furthermore, by \eqref{etabnd}, the shifted set $d_j' = g_{\kappa-j+1} - g_{1}$, $1\le j\le \kappa$  will satisfy \eqref{master2b} and therefore, $X_{\kappa} \in G$. It is clear from the construction that we have
\begin{equation} \label{rec2}
p^{\kappa}(x,G) \geq \left(\frac{\eta}{2\ep (2M+\kappa)}\right)^{\kappa}
\end{equation}

To prove (C2) holds, we first consider cylinder sets of the form $B = \{ d_i - d_{i+1} \in B_i \subset (\ep/2,\ep/2+\eta)$ for $i< \kappa$ and $d_i=0$ for $i \ge \kappa\}$. Then if $x \in G$, taking the same path that led to \eqref{rec2} yields the lower bound
\begin{equation} \label{lowerbnd}
p^{\kappa}(x,B) \geq \frac{|B_1|\cdots|B_{\kappa-1}|}{(2\ep)^{\kappa-1}}\left(\frac{1}{2M + \kappa}\right)^\kappa.
\end{equation}
If we let $\rho = $ Lebesgue measure on $G$ normalized to be a probability and recall that the Radon-Nikodym derivative $dp^\kappa(x,\cdot)/d\rho(\cdot)$ evaluated at a general measurable set $B$ can be written as the limit of $p^\kappa(x,B_k))/\rho(B_k)$ where $B_k$ is a sequence of cylinder sets, {\bf (C2)} follows.

To check positive recurrence, we let $\tau_A$ be the first hitting time of our regenerative set $G$. \eqref{rec} and \eqref{rec2} tell us that there is a positive constant $\theta = \theta(r,\ep)$ so that
$$
p^{2M+\kappa}(x,G) \geq \theta > 0
$$
for any $x \in \mathcal{S}$. Therefore, we have $E^x(\tau_A) \leq (2M+\kappa)/\theta < \infty$, completing the proof.
\end{proof}

The construction in the previous lemma also yields:

\begin{lemma} \label{alphaminlim}
$\alpha_{\min}(n) \rightarrow \infty$ a.s. as $n \rightarrow \infty$.
\end{lemma}

\begin{proof}
We can modify the construction in the previous Lemma to show that there exist constants $K, J \geq 1$, $\rho > 0$ so that
\begin{equation*}
P(\alpha_1((n+1)K) - \alpha_1(nK) \geq J \ep/2 | Y_{nK} = y) \geq \rho
\end{equation*}
for any $y \in \R^{+} \times \mathcal{S}$ and $n \geq 0$. Therefore, $\alpha_1(n) \rightarrow \infty$ a.s. by the Borel-Cantelli Lemma and the result follows since $\alpha_1(n) - \alpha_{\min}(n) < r$.
\end{proof}

\begin{theorem} \label{convtoeq}
As $n \rightarrow \infty$, $\|P^\alpha(\Delta_n \in \cdot) - \pi(\cdot)\|_{TV} \rightarrow 0$ for any initial $\alpha \in \R^{+}$.
\end{theorem}

\begin{proof}
It suffices to prove the result for the subsequences $n=m\kappa+j$ for $0 \le j < \kappa$, but then by using the Markov property at time $j$, it is enough to prove the result for $n=m\kappa$ and a general initial distribution.
To prepare for the proof, recall that one can modify the state space of a Harris recurrent Markov chain to have a point $\zeta$ that corresponds to being distributed on the set $A$ according to $\rho$ with the exact position being independent of the past.

To prove the result, we will construct a process $(\tilde{X}_n,\tilde{\Delta}_n) $ on $\mathcal{S} \times \mathcal{S}$ so that the marginal law of $\tilde{\Delta}_n$ is the law of $\Delta_{n\kappa}$, the marginal distribution of $\tilde{X}_n$ is $\pi$ for all $n$, and $\mathbb{P}(\tilde{X}_n \neq \tilde{\Delta}_n) \rightarrow 0$ as $n \rightarrow \infty$. Let $U_1,U_2, \ldots$ and $V_1,V_2,\ldots$ be independent and uniform on $[0,1]$. To begin, let $q_\alpha(x,\cdot) \equiv P(\Delta_{\kappa} \in \cdot |\Delta_0 = (\alpha,x))$ and
$$q(x,\cdot) \equiv \lim_{\alpha \to \infty} q_\alpha(x,\cdot) = p^\kappa(x,\cdot)$$
by Lemma \ref{convtoMC}. Define the function $J_n:\mathcal{S} \times [0,1] \rightarrow \mathcal{S}$ by
$$
P(J_n(x,U_n)\in B) = q_{\alpha_{\min}(n\kappa)}(x,B).
$$
Since $q_{\alpha_{\min}(n\kappa)}(x,\cdot) \in \mathcal{M}_1(\mathcal{S})$ and $\mathcal{S}$ is a separable metric space, defining $J_n$ is possible by Theorem 3.2 in Billingsley \cite{PB}. Suppose that $\tilde{X}_n$ has distribution $\pi$, define $Z_{n+1} = J_n(\tilde{X}_n,U_n)$ and
$$
\mu_n(A) \equiv P(Z_n \in A|\alpha_{\min}(n\kappa)) = \int q_{\alpha_{\min}(n\kappa)}(x,A) \pi(dx),
$$
and let $(\tilde{X}_{n+1},Z_{n+1})$ be a maximal coupling of $(X_n,Z_n)$ so that
$$
\mathbb{P}(\tilde{X}_{n+1} \neq Z_{n+1}) = \|\mu_n - \pi\|_{TV}
$$
(see, for example Thorisson \cite{HT}). Then from the definition of $\mu_n$ and $(\tilde{X}_{n+1},Z_{n+1})$ we have
\begin{align*}
\eta_{n+1}  & \equiv \mathbb{P}(\tilde{X}_{n+1} \neq Z_{n+1}) \\
& = \left\| \int q_{\alpha_{\min}(n\kappa)}(x,\cdot) \pi(dx) - \int q(x,\cdot) \pi(dx)\right\|_{TV} \rightarrow 0
\end{align*}
as $n \rightarrow \infty$ by Lemma \ref{convtoMC} and \ref{alphaminlim}.

When $\{\tilde{\Delta}_n=\tilde{X}_n\}$, we set $\tilde{\Delta}_{n+1} = J_n(\tilde{X}_n,U_n) = Z_{n+1}$ so that
$$
P(\tilde{X}_{n+1} \neq \tilde{\Delta}_{n+1}, \tilde{X}_n = \tilde{\Delta}_n) \leq \eta_{n+1}
$$
On $\{\tilde{X}_n \neq \tilde{\Delta}_n\}$, we take $\tilde{\Delta}_{n+1} = J_n(\tilde{X}_n,V_n)$. \eqref{lowerbnd} implies that $q(x,\zeta) \ge \lambda$, so it follows from Lemma \ref{convtoMC} that if  $\alpha_{\min}(n\kappa)\ge \alpha_0$ then
$q(\alpha_{\min}(n\kappa),x,\zeta) \ge \lambda/2$, and we have
$$
P(\tilde{X}_{n+1}=\tilde{\Delta}_{n+1}|\tilde{X}_n \neq \tilde{\Delta}_n) > \lambda/2
$$
so that if $\zeta_n = P(\tilde{X}_n \neq \tilde{\Delta}_n)$, then
$$
\zeta_{n+1} \leq (1-\lambda/2)\zeta_n + \eta_{n+1}.
$$
Iterating, yields the inequality
\beq \label{easysum}
\zeta_{n+1} \leq \sum_{i=1}^{n+1} (1- \lambda/2)^{n+1-i}\eta_i.
\eeq
Since $|1-\lambda/2|<1$ and $\eta_n \to 0$, the right hand side of \eqref{easysum} must also go to zero which yields
$$\|P(\Delta_{n\kappa} \in \cdot) - \pi(\cdot) \|_{TV} \leq P(\tilde{X}_n \neq \tilde{\Delta}_n) = \eta_n \to 0.$$
completing the proof. \end{proof}

It remains to prove the result on the linear growth of $\alpha_{\min}(n)$. Since $\alpha_j-\alpha_{min} \le r$, it suffices to establish this for $\alpha_{max}$. To do this, we look at the chains $Z_n = (X_n,U_n,V_n)$ with $U_n$ uniform on $[0,1]$ giving the index $k=\lceil N_nU_n \rceil$ of the value to be mutated, and $V_n$ independent uniform on $[-\ep,\ep]$ giving the change in the value due to mutation. It is clear that the distribution of $Z_n$ will converge in distribution to the product measure $\tilde{\pi} = \pi \times \hbox{uniform}[0,1] \times \hbox{uniform}[-\ep,\ep]$ so that if we let $f(Z_n) = \alpha_{\max}(n) - \alpha_{\max}(n-1)$ be the amount shifted at the $n^{th}$ step, then $f$ is non-negative and bounded above by $\ep$ so the strong law for functionals of Markov chains implies
\beq \label{limfunct}
\frac{\alpha_{\max}(n) - \alpha_{\max}(0)}{n} = \frac{1}{n} \sum_{m=1}^n f(Z_m) \to \int f(x) \tilde{\pi}(dx) =
\bar\alpha
\eeq
Since $f>0$ with positive probability, $\bar\alpha>0$. To extend this result to the real chain, let $(\tilde{X}_n,\tilde{\Delta}_n)$ be the coupled chain from the proof of Theorem \ref{convtoeq} and define $D_n=1$ if $\tilde{X}_n \neq \tilde{\Delta}_n$ and $D_n =0$ otherwise. From the proof of Theorem \ref{convtoeq}, we can dominate $D_n$ by a Markov Chain $B_n$ that has
\begin{align*}
P(B_{n+1}=1|B_n=0) &= \eta_{n+1} \\
P(B_{n+1}=0|B_n = 1) &= \frac{\lambda}{2}
\end{align*}
i.e., we can define the two processes on the same space so that $B_n \geq D_n$ for all $n$. Coupling $B_n$ with a homogeneous chain $B_n^\rho$ that has $P(B_{n+1}^\rho=1|B_n^\rho=0) = \rho$, $P(B_{n+1}^\rho=0|B_n^\rho = 1) = \lambda/2$, and stationary distribution $\pi^\rho$ with $\pi^\rho(1) = \rho/(\rho+\lambda/2)$ and recalling that $\eta_n \to 0$, it follows that
\beqx
\limsup_{n \to \infty} \frac{1}{n} \sum_{m=1}^n D_m \leq \limsup_{n \to \infty} \frac{1}{n} \sum_{m=1}^n B_m \leq \frac{\rho}{\rho + \lambda/2}.
\eeqx
Since this holds for any $\rho >0$, we must have
\beqx
\lim_{n \to \infty} \frac{1}{n} \sum_{m=1}^n D_m = 0
\eeqx
and the desired result now follows from \eqref{limfunct} and the fact that $0 \leq \alpha_{\max}(n) - \alpha_{\max}(n-1) \leq 1$ for all $n \geq 0$.

\section{Proof of Theorem \ref{alphafixedmain}} \label{alphafixedsec}

In this section, we prove Theorem \ref{deltafixedmain} for the DPEP as defined in Section \ref{deltafixedintro} and use the notation defined there. Since one of the keys to deriving our results will be comparison with a branching random walk, we continue adopting the perspective that $X_j(t)$ refers to the position of particle $j$ on positive half line. Note that if we set $\alpha_j = 1$, $k=N$ in \eqref{master}, we obtain the condition for coexistence:
\beq \label{master3}
\delta_N\left( \beta + \sum_{j=1}^N \left(1 - \frac{\delta_j}{\delta_N} \right) \right) < r.
\eeq

\n {\it Proof of Lemma \ref{alphafixedlem}}. Let $X_i(t) = - \log(\delta_i(t))$ and $X_1(t)> \cdots > X_{M}(t)$ be the rightmost $M$ particles at this time. It should be clear from (\ref{master3}) that if
\beq
e^{-X_M(T)} ( \beta + M ) < r
\label{Ntlb}
\eeq
then we will have $N_t \ge M$ for $t \ge T$. Let $y = - \log(r/(\beta+M))$. The right most particle is increasing in $t$. Since the number of particles changes by $\le 1$ each time and $\sum_{m=1}^\infty 1/m = \infty$ the right-most particle gives birth to the right of its current position plus 1/2 infinitely many times. Thus at some time $T$, we will have a point $\ge y+M$. Since $|X_i(t)-X_{i+1}(t)| \leq 1$ and points are only erased when \eqref{master3} fails, \eqref{Ntlb} follows. \eopt

\subsection{Asymptotics for $X_{max}$}

For the remainder of the paper, we let $Z_t$ be a branching random walk started from one particle at 0, in which particles give birth at rate 1 and displacements are uniform on $[-1,1]$. It is well known that the mean measure
\beq
EZ_t(A) = e^t P( S_t \in A)
\label{meanmeas}
\eeq
where $S_t$ is a continuous time random walk that jumps at rate one and takes step uniform on $[-1,1]$.
If we let $\phi(\theta) = (e^\theta - e^{-\theta})/2\theta$ be the moment generating function for the displacements, then
$$
E e^{\theta S_t} = \sum_{n=0}^\infty e^{-t} \frac{t^n}{n!} \phi^n(\theta) = \exp(t(\phi(\theta)-1))
$$
Chebyshev's inequality implies that if $\theta>0$
\beq
P( S_t > xt ) \le \exp( - t(\theta x - \phi(\theta) + 1) )
\label{ldub}
\eeq
and standard large deviations results imply that for $x \geq 0$,
\beq
\frac{1}{t} \log P( S_t > xt ) \to \Lambda(x) =  - \left( \sup_{\theta>0} \{\theta x - \phi(\theta)\} + 1 \right)
\label{ldasy}
\eeq
where $\Lambda(0) = 0$ and $\Lambda$ is strictly decreasing on $[0,\infty)$.

Biggins \cite{JB1}, Theorem 2 shows that the right-most particle in the branching random walk $Z_{\max}(t)/t \rightarrow a$ a.s. where $a$, defined in \eqref{ab}, is the smallest $x> 0$ such that $\Lambda(x)\le -1$. Since the particles $X_i(t)$ in our evolution model are a subset of those in the branching random walk, we have
$$
\limsup_{t\to\infty} X_1(t)/t \le a.
$$
The remainder of this section is dedicated to the proof of the lower bound
\beq \label{lowerbndgoal}
\liminf_{t\to\infty} X_1(t)/t \ge a.
\eeq

By Lemma \ref{alphafixedlem}, we know there exists some time $T$ so that $N_t \geq M$ for $t \ge T$. By the proof of Lemma \ref{alphafixedlem}, we can take $T$ to be the first time $e^{-X_M(T)} ( \beta + M ) < r$, which is a stopping time, so the future behavior of the process is not affected.

\begin{lemma} \label{edgetotoy}
If we start the toy model at time $T$ with positions equal to the rightmost $M$ particles at this time $X_1(T)> \cdots > X_{M}(T)$, then the $X_i(t)$ and $Y^M_i(t)$ can be defined on the same space so that $X_i(t) \ge Y^M_i(t)$ for all $1\le i \le M$ and $t \ge T$.
\end{lemma}

\begin{proof}
Couple the birth times of $X_i(t)$ and $Y^M_i(t)$ and the displacements of their offspring. Recall that if a birth from $X_k(t)$ with $k > M$ lands to the right of some $X_i(t)$, $i \leq M$, we renumber the $X_i$ and put them in decreasing order. Births of particles from $X_k(t)$ for $k>M$ may cause the $X$'s to get ahead of the $Y$'s, but coupled births for $i \le M$ cause the vectors of $X$'s and $Y$'s to move in parallel and the desired comparison follows.
\end{proof}

For our next comparison consider the branching random walk started with one particle at $Y^M_1(0)$. Let $T_k$ be the time of the $k$th birth, with $T_0=0$, and for $t \in [T_{k-1},T_k)$ let $\zeta^k_1(t) > \zeta^k_2(t) > \cdots > \zeta^k_k(t)$ be the locations of the particles present.

\begin{lemma} \label{toytobrw}
We can couple the branching random walk and the toy model so that for $t \in [T_{k-1},T_k)$, $Y^M_j(t) \ge \zeta^k_j(t)$ for $1\le j \le k$ and $k<M$.
\end{lemma}

\begin{proof}
Couple the birth times of $\zeta^k_j(t)$ and $Y^M_j(t)$ for $j\le k$ and $t \in (T_{k-1},T_k]$, i.e., there will be no births in $(T_{k-1},T_k)$ and the same particle will give birth at time $T_k$. Births of particles from $Y_j(t)$ for $j>k$ may cause the $Y$'s to get ahead of the $\zeta$'s, but coupled births for $j \le k$ cause the vectors of $\zeta$'s and $Y$'s to move in parallel.
\end{proof}

\begin{lemma} \label{toy2}
Let $B_M$ be the time of the $M$th birth in the branching random walk.
\beqx
\liminf_{t \rightarrow \infty} \frac{Y_1^M(t)}{t} \ge \frac{EZ_{max}(B_M)}{EB_M} \to a \quad \text{as $M\to\infty$}
\eeqx
\end{lemma}

\begin{proof}
Let $T_{k,1} = T_k$ where $T_k$ are as in Lemma \ref{toytobrw} and for $j >1$, let $T_{k,j}$, $k \leq M$ denote the time of the $k^{th}$ birth in a BRW started with a single particle at $Y_1^M(T_{M,j-1})$ at time $T_{M,j-1}$ and let $\zeta_{M,j}$ denote the position of the rightmost particle at time $T_{M,j}$. Repeatedly applying the comparison in Lemma \ref{toytobrw} and letting yields
\begin{align*}
\frac{Y_1^M(t)}{t} &\geq \frac{\sum_{j: \, T_{M,j} \leq t} (Y_1^M(T_{M,j+1}) - Y_1^M(T_{M,j}))}{t} \\
&\geq  \frac{\sum_{j: \, T_{M,j} \leq t} (\zeta_{M,j+1} - Y_1^M(T_{M,j}))}{t}.
\end{align*}
But the time intervals $T_{M,j+1}-T_{M,j}$ are iid with mean $E B_M$ so the first part of the result follows from the renewal theorem. To prove the second part, we note that Biggins' result implies
$$
Z_{\max}(B_M)/B_M \rightarrow a \quad\hbox{almost surely.}
$$
Since $B_M = \xi_1 + \cdots + \xi_M$ where the $\xi_i$ are independent exponentials with mean $1/i$, it is easy to see that $B_M/EB_M \to 1$, so
$$
Z_{\max}(B_M)/EB_M \rightarrow a \quad\hbox{almost surely.}
$$
Therefore, the result will follow from the dominated convergence theorem if we can show that
\beqx
E\left( \sup \frac{Z_{\max}(B_M)}{EB_M} \right)< \infty.
\eeqx
By the Cauchy Schwarz inequality, it suffices to show
\beq \label{goal1}
E \left(\sup_{t \geq 1}\frac{Z_{\max}(t)}{t}\right)^2 < \infty
\eeq
and
\beq \label{goal2}
E \left(\sup \frac{B_M}{EB_M}\right)^2 < \infty.
\eeq
To prove \eqref{goal1}, we note that \eqref{meanmeas} and \eqref{ldub} imply that
$$
P( Z_{\max}(t) > x t ) \le e^{t(1+\Lambda(x))}
$$
and since $\Lambda$ is concave with $\Lambda(0)=0$ and $\Lambda(a)=-1$ with $a< 1$, it follows that for $x\ge 1$
$$
P( Z_{\max}(t) > x t ) \le e^{t(1-x)}
$$
Now if $Z_{\max}(t)/t > 2x $ for some $t$, then since $Z_{\max}(t)$ is non-decreasing, we must have $Z_{\max}(s)/s > x$ for some $s \in [t,t+1]$ and therefore, summing over all integers $t$ from 1 to $\infty$, we see that if $x >2$
\beqx
P\left( \sup_{t \ge 1} Z_{\max}(t) / t  > 2x \right) \le e^{1-x}
\eeqx
which proves \eqref{goal1}. To prove \eqref{goal2}, we note that $EB_M = \sum_{i=1}^M 1/i$ and
$$
E\exp(\theta B_M ) = \prod_{i=1}^M \frac{1}{1-\theta/i}
$$
for $0 < \theta < 1$, so by Chebyshev's inequality,
$$
P( B_M > y EB_M ) \le \exp\left( -\theta y \sum_{i=1}^M \frac{1}{i} - \sum_{i=1}^M \log(1 - \theta/i) \right)
$$
Taking $\theta=1/2$ and choosing $c$ so that $\log(1-x) \ge - x - c x^2$ when $0< x < 1/2$, we have
\begin{align*}
P( B_M > y EB_M ) & \le \exp\left( \sum_{i=1}^M \frac{1}{2i} (1-y)  + \frac{c}{4 i^2} \right) \\
& \le C \exp\left( \frac{1-y}{2} \log (M+1) \right) = C (M+1)^{(1-y)/2}
\end{align*}
Therefore if $y>3$,
\beqx
\sum_{M=2}^\infty (M+1)^{(1-y)/2} \le \int_2^\infty x^{(1-y)/2} \, dy = \frac{2^{(3-y)/2}}{(y-3)/2}
\eeqx
which yields \eqref{goal2}, completing the proof.
\end{proof}

\eqref{lowerbndgoal} follows from Lemma \ref{edgetotoy} and \ref{toy2} which completes the proof that the speed of the rightmost particle is $a$. We shall complete the proof of Theorem \ref{alphafixedmain} in the next section by showing the speed of the leftmost particle is $b$, but first we pause to prove Corollary \ref{fixdeltspeed}.

\mn
{\it Proof of Corollary \ref{fixdeltspeed}.} Suppose we choose $\ep$ small enough so that $\ep M(M-1)/2 < r$. Using the coupling in Lemma \ref{toytobrw} we can use the particles $\zeta_j^k$, $j \le k \le M$, from the branching random walk started at $X_{max}$ to get a lower bound on the right-most $k \le M$ particles in the predator evolution with fixed $\delta$. An induction argument shows that the spacings between the corresponding particles in the predator evolution are $\le \ep$ at all times. Since we have assumed $\ep \sum_{j=1}^{M-1} j < r$ the right-most $k \le M$ particles are never killed. The remainder of the proof is the same as before. \eopt

\subsection{Asymptotics for $X_{min}$}

In order to get the speed of the leftmost particle, we will need the following result on a branching random walk with killing which is an adaptation of Biggins \cite{JB1}, Theorems 1 and 2, which proves this result without killing.

\begin{lemma} \label{killedBRW}
Let $Z_t(\gamma,A)$ denote the number of particles in $A$ under a branching random walk with birth rate one, displacements uniform on $[-1,1]$, killing to the left of $-K+\gamma t$, and started with one particle at 0. Then for any $c > \gamma$ on the set of nonextinction
\beq \label{ldp}
\lim \frac{1}{t} \log Z_t(\gamma,[ct,\infty)) = I(c)
\eeq
where $I(c) = 1+\Lambda(c)$, and the probability of extinction tends to 0 as $K\to\infty$.
\end{lemma}

\begin{proof}
Theorem 2 in Biggins along with (\ref{meanmeas}) and (\ref{ldasy}) yields \eqref{ldp} in the case of no killing and since $Z_t(\gamma,[ct,\infty)) \subset Z_t([ct,\infty))$, we get the upper bound in \eqref{ldp}. To get the lower bound, we recall that to prove the corresponding lower bound for the process without killing, Biggins lets $Z^k_{m+1}$ be the points at time $(m+1)k$ that are at least $k c$ units to the right of their ancestor in $Z^k_{m}$ at time $mk$. $|Z^k_m|$ is a branching process with offspring distribution $|Z^k_1|$ so $(|Z^k_m|)^{1/m} \to E|Z^k_1|$ on the nonextinction set. Combining (\ref{meanmeas}) and (\ref{ldasy}) implies $(1/k) \log E|Z^k_1| \to I(c)$ which yields the desired lower bound.

To extend this construction to the process with killing, let $\bar Z^k_{m+1}$ be the points at time $(m+1)k$ that are at least $k c$ units to the right of their ancestor in $\bar Z^k_{m}$ at time $mk$ and are not killed by going to the left of $-K + \gamma t$ of $mk \le t \le (m+1)k$. By construction, all points in $\tilde{Z}_m^k$ are $ \geq cmk$ and we have chosen $\gamma < c$ so for large $m$, the killing has little effect and on the set of non-extinction we have
$$
\frac{1}{m} \log |\bar Z^k_m| \to \log E|Z^k_1|.
$$
Using (\ref{meanmeas}) and (\ref{ldasy}) again gives the desired lower bound.
\end{proof}

With this result in hand, we can complete the

\mn
{\it Proof of $X_{min}(t)/t \to b$.} When $X_{min}(t)$ increases we must have
$$
N_t e^{-X_{\min}(t)} \geq r.
$$
Since the particles in $X$ are a subset of the particles in the branching random walk, it follows that if $X_{\min}(t) \geq (b+\ep)t$,
$$
N_t e^{-X_{\min}(t)} \leq Z_t([(b+\ep)t,\infty))e^{-(b+\ep)t }\to 0
$$
as $t \rightarrow \infty$ since $I(c) < c$ for all $c > b$. Therefore, $\limsup X_{\min}(t)/t \leq b$ a.s.

To prove that  $\liminf X_{\min}(t)/t \geq b$ a.s., let $c \in (b,a)$ and $\ep > 0$. Choose $K$ large enough so that the probability of extinction in the branching random walk with killing at $-k + bt$ is less than $\ep$ for all $k \geq K$ and then take $T$ large enough so that $X_1(t)\geq c t$, for all $t \geq T$ (which is possible since $\lim X_1(t)/ t = a$) and so that $b T > K$. Suppose that $ X_{\min}(t) \leq (b- \rho)t$ for some $\rho > 0$. Then by comparing with a branching random walk with killing at $-X_1(T)+b t$, we have
\begin{align}
F(t) &:= e^{-X_{\min}(t)} \sum_{j=1}^{N_t} (1-e^{-X_j(t)/X_{\min}(t)}) \label{thisguy} \\
& \geq e^{-(b-\ep)t} (1-e^{-(c-b+\ep)t})Z_t(b,[ct,\infty)). \nonumber
\end{align}
But on the non-extinction set (which has probability at least $1-\ep$), we have
\begin{align*}
\lim \frac{1}{t} \log [e^{-(b-\ep)t} (1-e^{-(c-b+\ep)}t)Z_t(\gamma,[ct,\infty))] = I(c) - b + \ep \to \ep > 0
\end{align*}
as $c \downarrow b$ and therefore, we must have $ X_{\min}(t) > (b- \ep)t$ eventually or there would exist a sequence of points $t_i \to \infty$ for which $F(t_i) \to \infty$, contradicting \eqref{contscond}. Therefore, $P(\liminf X_{\min}(t)/t < b) < \ep$ and since $\ep$ is arbitrary, this proves the result.

To conclude that $\liminf_{t\to\infty} (\log N_t)/t \ge b$ a.s., note that if $\ep>0$ then for large times there are at least $\exp((I(c)-\ep)t)$ points of $X$ to the left of $ct$. Picking $c$ close to $b$ and $\epsilon$ small gives the desired result.  \eopt

\clearpage

%\begin{figure}[tbp]
%  \centering
%\includegraphics[width=2.5in,keepaspectratio]{preyevole}
% \caption{Simulation of prey evolution with $\ep=0.01$ when $\alpha=4$, $\beta=2$, and $\delta=\gamma=1$.}
%  \label{preyevolvefig}
%\end{figure}

\begin{figure}[tbp] % float placement: (h)ere, page (t)op, page (b)ottom, other (p)age
  \centering
  % file name: F:/predprey/2Dscatter.eps
  \includegraphics[width=3in,keepaspectratio]{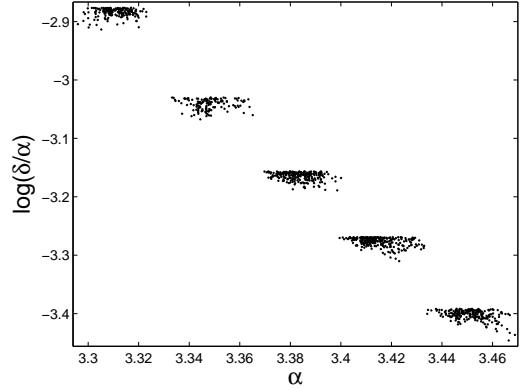}
\caption{The five clusters, from upper left to lower right, are the characteristics of the coexisting predators in a sample run of the Predator EP after $10^4, 1.25 \times 10^4, 1.5\times 10^4, 1.75\times 10^4$ and $2\times 10^4$ mutations have occurred. The consumption rates, $\alpha$, of all coexisting predators are plotted on the $x$-axis and the corresponding values of $\log \ell = \log(\delta/\alpha)$ are plotted on the $y$-axis. Parameters: $r=1$, $\alpha(0)=3$, $\delta(0)=.45$, $\ep=.01$.}
\label{2Dscatter}
\end{figure}

\begin{figure}[tbp] % float placement: (h)ere, page (t)op, page (b)ottom, other (p)age
  \centering
  % file name: F:/predprey/2Dnumber.eps
  \includegraphics[height = 1.6 in, width=3 in]{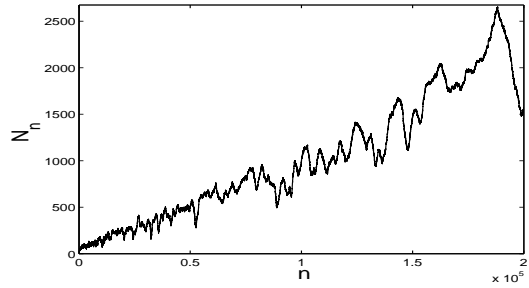}
\caption{Plot of $N_n$ = number of coexisting species in the population after the $n^{th}$ mutation has occurred in the Predator EP from Figure \ref{2Dscatter}.}
\label{2Dnumber}
\end{figure}

\begin{figure}[tbp] % float placement: (h)ere, page (t)op, page (b)ottom, other (p)age
  \centering
  % file name: F:/predprey/eps01alp3.eps
  \includegraphics[height = 2.2in, width=3in]{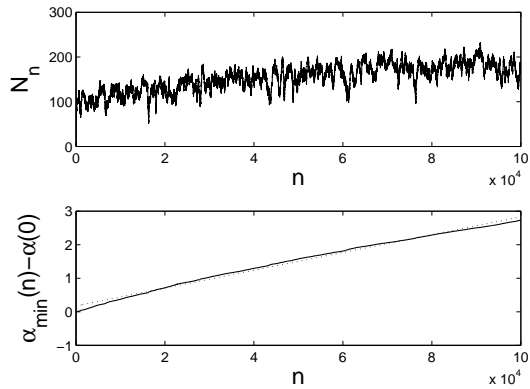}
\caption{Top: Plot of the number of coexisting predators at time $n$ in the APEP with $\ep = .01,$ $r=1$, $\alpha(0)=3$. Bottom: Plot of the change in $\alpha_{\min}(n)$ for the same simulation.}
\label{eps01alp3}
\end{figure}

\begin{figure}[tbp] % float placement: (h)ere, page (t)op, page (b)ottom, other (p)age
  \centering
  % file name: F:/predprey/verysmalleps2.EPS
  \includegraphics[height = 2.2in, width=3in]{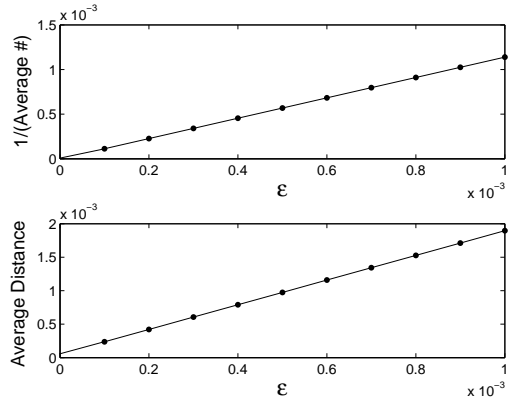}
\caption{The top panel shows the inverse of the average number of species and the bottom panel shows the average maximum distance between $\alpha$'s as a function of $\ep$ for the APEP. Here, we have run one simulation for each value of $\ep=.001,.002,\dotso,.01$ with $r=1$, $\alpha(0)=3$ and then averaged out the results of each simulation over the last 25,000 time steps to obtain the values for the plotted points. The solid lines are the corresponding least square lines. It appears that the number of coexisting species is $O(1/\ep)$ and the maximum distance between coexisting types is $O(\ep)$ as $\ep \to 0$.} \label{verysmalleps}
\end{figure}

\begin{figure}[tbp] % float placement: (h)ere, page (t)op, page (b)ottom, other (p)age
  \centering
  % file name: F:/predprey/alpden.eps
  \includegraphics[height = 1.6 in, width=3 in]{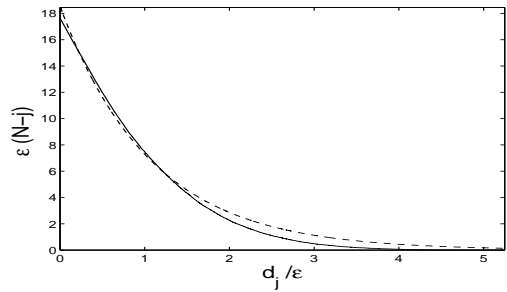}
\caption{Plot of the distribution of predator types for a single run of the APEP at time $n=50,000$ with $\ep=.001$ and $r=1$. The solid line connects the points $(d_j/\ep, \ep(N_n-j))$, $j=1,2,\dotso N_n = 17626$. The dashed line gives an exponential approximation.}
\label{alphaden}
\end{figure}

\begin{figure}[tbp]
\centering
\includegraphics[height = 1.6 in, width=3 in]{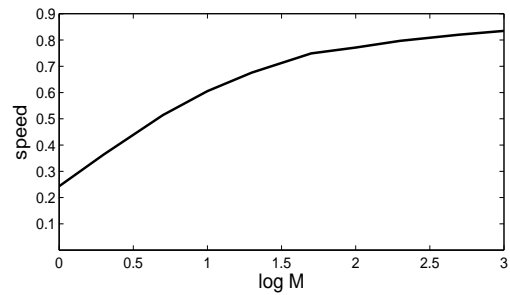}
\caption{Graph of the speed $a_M$ versus $\log M$ showing slow convergence to the limit $a \approx 0.9053$ for the finite branching-selection particle system $Y^M$ defined in Section \ref{fixedalphaintro}.}
\label{aMslow}
\end{figure}

\begin{figure}[tbp] % float placement: (h)ere, page (t)op, page (b)ottom, other (p)age
  \centering
  % file name: F:/predprey/Xdens.eps
  \includegraphics[height = 1.8 in, width=3 in]{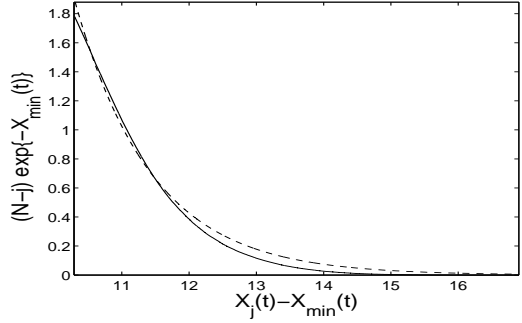}
\caption{Plot of the distribution of predator types for a single run of the DPEP with $r=1$ at time $t \approx 20.25$ (after $n=50,000$ insertions). The solid line shows the point $(X_j(t)-X_{\min}(t),(N_t-j) e^{-X_{\min}(t)})$, $j=1,...,N_t =25467$. The dashed line gives an exponential approximation.}
\label{Xdens}
\end{figure}

\begin{figure}[h] % float placement: (h)ere, page (t)op, page (b)ottom, other (p)age
  \centering
  % file name: F:/predprey/preyinvadepic2.EPS
  \includegraphics[height = 2.25 in , width=4 in]{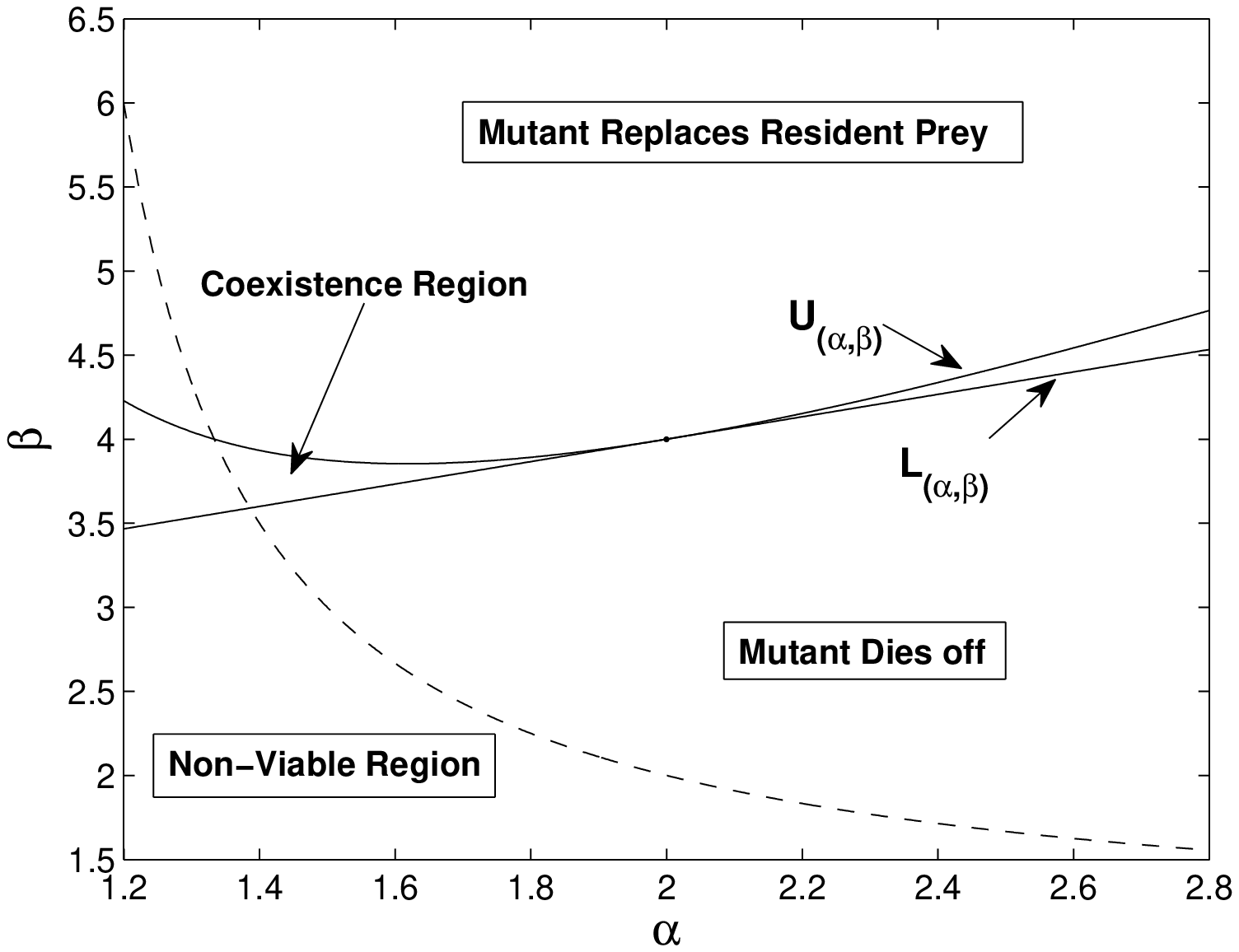}
\caption{Plot of the invadability curves for $(\alpha_1, \beta_1) = (2,4)$. For the predator, we set $\delta =1$. The dashed line shows the boundary of the viable region.}
\label{preyinvadepic}
\end{figure}


\begin{thebibliography}{99}


\bibitem{AD}  Anagnost, J., and Desoer, C. (1991) An Elementary Proof of the Routh-Hurwitz Stability Criterion \emph{Circuits Systems Signal Process}
10, No. 1


\bibitem{AN} Athreya, K., and Ney, P.E. (1978) A new approach to the limit theory of recurrent Markov chains {\it Trans. Amer. Math. Soc.} 245, 493--501

\bibitem{JB1} Biggins, J.D. (1977) Chernoff's theorem in branching random walk. {\it J. Appl. Probab.} 14, 630--636


\bibitem{JB2} Biggins, J.D. (1979) Growth rates in branching random walks. {\it Z fur Wahr. verw Gebiete} 48, 17--34


\bibitem{PB} Billingsley, P. (1971) {\it Weak Convergence of Measures.}, SIAM, Philedelphia, PA.

\bibitem{BG} Berard, J. and Gouere, J. (2008) Brunet-Derrida behavior of branching-selection particle systems on the line {\it arXiv:0811.2782v1}

\bibitem{TB} Brown, T. (1978) A martingale approach to the Poisson convergence of simple point processes. {\it Ann. of Prob.}  6, 615--628

\bibitem{BD} Brunet, E. and Derrida, B. (1997) Shift in the velocity of a front due to a cutoff. {\it Phys. Rev. E.} 56, 2597–-2604

\bibitem{CFM} Champagnat, N., Ferri\`ere, R., and M\'el\'eard, S. (2008) From individual stochastic processes to macroscopic models in adaptive evolution. \emph{Stoch. Models}, 24 Suppl. 1, 2--44.

\bibitem{CL} Champagnat, N. and Lambert, A. (2007) Evolution of discrete populations and the canonical diffusion of adaptive dynamics. {\it Ann. Appl. Prob.}  17,  102--155

\bibitem{CM} Champagnat, N. and M\'el\'eard, S. (2009) Polymorphic evolution sequence and evolutionary branching. (Submitted)

\bibitem{DIR} Dercole, F., Irisson, J.O., and Rinaldi, S. (2002) Bifurcation analysis of a predator-prey coevolution model.
{\it SIAM J. Appl. Math.} 63, 1378--1391

\bibitem{DL} Dieckmann, U. and  Law, R. (1996) The dynamical theory of coevolution: a derivation from stochastic ecological processes. {\it J. Math. Biol.} 3, 579--612

\bibitem{DML}  Dieckmann, U., Marrow, P., and Law, R. (1995) Evolutionary cycling in predator-prey interactions:
Population dyanmics and the Red Queen. {\it J. theor. Biol.} 176, 91--102

%\bibitem{RD1} Durrett, R. (1992) Predator-prey systems.
%Pages 37-58 in {\it Asymptotic problems in probability theory: stochastic models and
%diffusions on fractals.}  Edited by K.D. Elworthy and N. Ikeda.
%Pitman Research Notes 283. Longman, Essex, England

\bibitem{RD2} Durrett, R. (2002) Mutual invadability implies coexistence in spatial models. {\it Memoirs of
the AMS.} Volume 156, Number 740

\bibitem{RD3} Durrett, R. (2005) {\it Probability: Theory and Examples.}, Third edition. Duxbury Press, Belmont, California.

\bibitem{DR} Durrett, R. and Remenik, D. (2009) Brunet-Derrida particle systems, free boundary problems and Wiener-Hopf equations. Submitted to {\it Ann. Prob.}

\bibitem{EK} Ethier, S., and Kurtz, T. (1986) {\it Markov Processes: Characterization and Convergence.}
John Wiley and Sons

%\bibitem{GM2} Geritz, S. A. H., Kisdi, É., Meszéna, G., and Metz, J. A. J.. (1998) Evolutionarily singular strategies and the adaptive growth and branching of the evolutionary tree. \emph{Evol. Ecol.} \textbf{12}, 35--57.

\bibitem{Ger} Geritz, S.A.H. (2005). Resident-invader dynamics and the coexistence of similar strategies. \emph{J. Math. Biol.} 50, 67--82..


\bibitem{GK} Geritz, S.A.H., M. Gyllenberg, F.J.A. Jacobs, and K. Parvinen. (2002) Invasion dynamics and attractor inheritance. \emph{J. Math. Biol.} 44 548--560.


\bibitem{GM} Geritz, S.A.H., Metz, J.A.J, Kisdi, E., and Meszena, G. (1997) Dynamics of Adaptation and Evolutionary Branching. {\it Phys. Rev. Letters} 78, 2024--2027.


\bibitem{HSA} Hofbauer, J. and Sigmund, K. (1990) Adaptive dynamics and evolutionary stability. {\it Appl. Math. Lett.}, 3(4), 75--79.

\bibitem{HS} Hofbauer, J. and Sigmund, K. (1998) {\it Evolutionary Games and Replicator Dynamics.}, Cambridge University Press, Cambridge, United Kingdom.

\bibitem{JE} Jones, L. E. and Ellner, S.P. (2007) Effects of rapid prey evolution on predator-prey cycles.
{\it J. Math. Biol.} 55, 541--573

\bibitem{HK} Kesten, H. (1978) Branching Brownian motion with absorption. {\it Stoch. Proc. Appl.} 7, 9--47

\bibitem{MNG} Metz, J.A.J., Nisbet, R.M., and Geritz, S.A.H. (1992) How should we define ``fitness'' for general ecological scenarios? \emph{Trends Ecol. Evol.} 7, 198--202.

\bibitem{JM} Murray, J.D. (1989) {\it Mathematical Biology.} Springer-Verlag, Berlin

%\bibitem{Neu} Neuhauser, C. (1992). Ergodic theorems for the multitype contact process. \emph{Probab. Theory
%Related Fields.} 91, 467–-506.

\bibitem{Tak} Takeuchi, Y. (1996) {\it Global Dynamical Properties of Lotka-Volterra Systems}, World Scientific Publishing Co., Singapore.
\bibitem{HT} Thorisson, H. (1986) On maximal and distributional coupling. {\it Ann. Appl. Prob.}  44,  873--876

\end{thebibliography}
\end{document}